\documentclass[12pt]{amsart}
\usepackage[usenames,dvipsnames]{xcolor}
\usepackage{tikz}
\usepackage{fancyhdr}
\usepackage{txfonts}
\usepackage{graphicx}
\usepackage{epsfig}
\usepackage{mathrsfs}
\usepackage{amssymb}
\usepackage{latexsym}
\usepackage{amsmath} 
\usepackage{cancel}

\usepackage{indentfirst}
\allowdisplaybreaks
\usepackage{amsfonts}
\usepackage{amsbsy}
\usepackage{amscd}
\usepackage{subfigure}
\usepackage{verbatim}
\usepackage{color,tikz}
\usetikzlibrary{calc}

\usepackage{amsthm}
\usepackage{amsfonts}
\usepackage{amsbsy,calc}
\usepackage{graphicx,ifthen,cite}

\usepackage{booktabs}
\usepackage{graphicx}

\usepackage[normalem]{ulem}

\newcommand{\A}{{\mathcal A}}
\newcommand{\B}{{\mathcal B}}
\newcommand{\D}{{\mathcal D}}

\renewcommand{\P}{{\mathcal P}}

\newcommand{\R}{{\mathbb R}}
\newcommand{\Z}{{\mathbb Z}}
\newcommand{\C}{{\mathbb C}}

\newcommand{\N}{{\mathbb N}}

\newcommand{\la}{\lambda}
\newcommand{\La}{\Lambda}
\newcommand{\Ga}{\Gamma}

\newtheorem{prop}{Proposition}[section]
\newtheorem{lem}[prop]{Lemma}

\newtheorem{defi}[prop]{Definition}
\newtheorem{coro}[prop]{Corollary}
\newtheorem{theo}[prop]{Theorem}

\newtheorem{exam}[prop]{Example}

\newtheorem{problem}{Problem}
\newtheorem{conjecture}{Conjecture}

\newcommand{\dev}[2][]{{\,\textup{d}}^{#1}#2}

\def\dmu{\dev\mu}

\newif\ifdraft
\drafttrue
\ifdraft
\else
\fi
\numberwithin{equation}{section}
\title{}
\author{}
\date{}

\begin{document}
\baselineskip 16pt

\title[Characterization on spectra of self-similar measures and the dual spectral set conjecture]{Characterization on spectra of  self-similar measures and the dual spectral set conjecture}

\author[L.-X. An]{ Li-Xiang An}
\address{School of Mathematics and Statistics, Key Laboratory of Nonlinear Analysis \& Applications (Ministry of Education), Central China Normal University, Wuhan 430079, P. R. China.}
\email{anlixianghai@163.com;}

\author[Y.-S. Fu]{Yan-Song Fu$^*$}\address{Department of  Mathematics   \\ China University of Mining and Technology (Beijing) \\ Beijing, 100083\\P. R. China.}\email{yansong$\_$fu@126.com;}

\author[M.-X. Jiang]{Ming-Xuan Jiang$^*$}
\address{School of Mathematics and Statistics, Key Laboratory of Nonlinear Analysis \& Applications (Ministry of Education), Central China Normal University, Wuhan 430079, P. R. China.}
\email{mingxuanjiang1212@163.com.}

\thanks{*Corresponding author.}
\thanks{L.-X. An and  M.-X. Jiang are supported by the National Natural Science Foundation of China (12171181 and 12371087). Y.-S. Fu is supported by the National Natural Science Foundation of China (12371090).} 
\def\subjclassname{2020 Mathematics Subject Classification}
\subjclass[2020]{Primary 28A80, 42A65,52C22}
\date{}

\begin{abstract}
A discrete set $\Lambda$ is called a {\it spectrum} of a Borel probability measure $\mu$ if  the exponential functions $\{e^{2\pi i \langle\la, x\rangle}:\la\in\Lambda\}$ form an orthonormal basis for $L^2(\mu)$. In this work, we  first give necessary and sufficient conditions for a self-replicating translation  set to be a spectrum  of the self-similar  measure associated with a product-form Hadamard triple, which were discovered by the first-named author and Lai [Adv. Math. 431 (2023)  Paper  No.109257]. These results extend the
studies by {\L}aba and Wang [J. Funct. Anal. 193 (2002) 409-420], Dutkay and Lai [J. Math. Pures Appl. 107 (2017) 183-204] in $\R$. As an application, we can explicitly construct such  a self-replicating  spectrum. Finally, we prove that the dual spectral set conjecture holds for a self-replicating translation set. 
\end{abstract}

\maketitle

\section{Introduction}

\subsection{Background}
A Borel probability measure $\mu$ on $\R^d$ is called a \emph{spectral measure} if there exists a discrete set $\La\subset\R^d$, known as a {\it  spectrum} of $\mu$, such that the family of complex exponential functions  $$E(\Lambda):=\left\{e^{2\pi i \langle\la, x\rangle}:\lambda\in\Lambda\right\}$$ forms an orthonormal basis for  $L^2(\mu)$.  $(\mu, \Lambda)$ is referred to as a \emph{spectral pair}.

Spectral measures are a natural generalization of spectral sets. A measurable set $\Omega$ in $\R^d$ with positive and finite measure is called a {\it spectral set} if $L^2(\Omega)$ has an orthonormal basis $E(\Lambda)$. Spectral sets have been studied rather extensively, particularly over the past fifty years. (A partial list of these studies can be found in the references of this paper.) The
major unsolved problem concerning spectral sets is the following conjecture of Fuglede\cite{Fug1974}:
\begin{conjecture}[The spectral set conjecture]
Let $\Omega \subset \mathbb{R}^d$ be a set with positive and finite Lebesgue measure. Then  $\Omega$  is  a   spectral set if and only if  $\Omega$  is  a translational tile. 
\end{conjecture}
$\Omega$ is  called a {\it translational tile} if there exists a discrete set ${\mathcal J}\subset \R^d$ such that
$$\Omega+{\mathcal J}=\R^d\quad \text{and} \quad (\Omega+t)^{\rm o}\cap (\Omega+t')^{\rm o}=\emptyset, ~ t\not=t'\in \mathcal J,$$ in which case  ${\mathcal J}$ is  called a {\it tiling set} and $(\Omega, {\mathcal J})$  is  a {\it tiling pair}. 
Dually,  Jorgensen and Pedersen  \cite{JP1999} in 1999 proposed the following conjecture:
\begin{conjecture}[The dual spectral set conjecture]
Let $\Lambda\subset\R^d$. Then $\Lambda$ is a spectrum if and only if $\Lambda$ is a tiling set, i.e., there exists a set $\Omega$ so that $(\Omega, \Lambda)$ is a spectral pair if and only if there exists a set $\Omega'$ so that $(\Omega', \Lambda)$ is a tiling pair.
\end{conjecture}

Nowadays we know  that the conjectures are not true without imposing extra
assumptions on the domain or the frequency.  Fuglede's spectral set conjecture was first disproved by Tao  \cite{Tao2004} in $\R^d$ for $d\geq5$, showing that spectrality does not imply tile. It was further disproved by Kolountzakis and Matolcsi \cite{KM2006,KM2006-1}, who showed that the conjecture fails in both directions for dimensions $d \geq 3$. 
It is still open in one and two dimensions, and the analogous conjectures have  been extensively investigated on convex bodies \cite{LevMat2022}, $p$-adic field $\mathbb{Q}_p$ \cite{FFS2016,FFLS2019}, and finite Abelian groups \cite{FKS2022,IosMayPak2017,M2022,S2020,Zh2023}.

 The aim of this paper is to study Fuglede's problem in the setting of self-similar measures supported on self-similar sets.  Let $N\geq2$ be a positive integer and $\mathcal{D}\subset\mathbb{Z}$ be a finite set.  We consider the following \emph{iterated function system} (IFS)
$$\Phi:=\left\{\phi_{d}(x)=N^{-1}(x+d):d\in\mathcal{D}\right\}.$$
Hutchinson \cite{Hut1981} proved that there is a unique non-empty compact set $T(N,\mathcal{D})$, called the {\it attractor} or {\it self-similar set} of  $\Phi$, such that
   $$T(N,\mathcal{D})=\bigcup_{d\in\mathcal{D}}\phi_{d}(T(N,\mathcal{D})).$$ 
Moreover, he showed that there exists a unique Borel probability measure $ \mu_{N,\mathcal{D}}$ (with equal weights) on the attractor $T(N,\mathcal{D})$ such that
$$\mu_{N,\mathcal{D}}(E)=\frac{1}{\#\mathcal{D}} \sum_{d\in\mathcal{D}}\mu_{N,\mathcal{D}}\left(\phi^{-1}_{d}(E)\right),\quad\forall E\;  \text{Borel}.$$
The measure $\mu_{N,\mathcal{D}}$ is called the \emph{self-similar measure} of the IFS  $\Phi$. 


In 1998,  Jorgensen and Pedersen \cite{JP1998} proved that the middle-fourth Cantor measure $\mu_{4,\{0,2\}}$ has a spectrum 
$$\Lambda(4, \{0, 1\}):=\left\{\sum_{j=0}^k4^j\ell_j: \ell_j\in\{0, 1\} \text{ for each $j$ and } k\geq0\right\},$$ which established the first example of a singular continuous fractal spectral measure.  Strichartz \cite{Str1998,Str2000} and later {\L}aba and Wang \cite{LW2002} extended this result to the case of self-similar measures. In their work, the concept of Hadamard triples plays an important role. 
 Recall that,  for $1<N\in\N$ and digit sets $\D, {\mathcal L}\subset\Z$  with $\#\D=\#{\mathcal L}$, the triple $(N, \mathcal{D}, \mathcal{L})$  is called a \textit{Hadamard triple} if the matrix
$$\frac{1}{\sqrt{\#\mathcal{D}}}\left(e^{-2\pi i \frac{d \ell}{N}}\right)_{d\in\mathcal{D}, \ell\in\mathcal{L}}$$
is  unitary. Using the Hadamard triple  $(N, \mathcal{D}, \mathcal{L})$  and the associated Ruelle transfer operator
\begin{equation}\label{R-1}
   \mathcal{R}(f)(\xi)=\sum_{\ell\in\mathcal{L}} \left|M_\mathcal{D}\left(\phi_{\ell}(\xi)\right)\right|^2 f\left(\phi_{\ell}(\xi)\right),
\end{equation}
where
$ M_\mathcal{D}(\xi):=\frac{1}{\#\mathcal{D}}\sum_{d\in\mathcal{D}}e^{-2\pi id\xi},$
$\phi_{\ell}(\xi):=N^{-1}(\xi+\ell), ~\ell \in \mathcal{L},$
{\L}aba and Wang \cite{LW2002} showed that, under the assumption $\gcd(\mathcal{D})=1$,  the following \textit{self-replicating translation set} 
$$\Lambda(N, {\mathcal L}):=\left\{\sum_{j=0}^k N^j\ell_j: \ell_j\in{\mathcal L} \text{ for each $j$ and } ~~k\geq 0\right\}$$
forms a spectrum of $\mu_{N, \D}$ if there exists no nonzero integral cycle point for ${\mathcal L}$ (see Subsection \ref{sec.1.main} for the definition). Subsequent  studies by  Dutkay and   Jorgensen\cite{DJ2006,DJ2007,DJ2009} and Dutkay and Lai \cite{DL2017} extended this result to higher dimensions: if $\Lambda$ is the smallest set containing all integral cycle points and satisfying ${\mathcal L}+N\Lambda\subset\Lambda$, then $\Lambda$ is a spectrum of $\mu_{N, \D}$.  To date,  the characterization of the spectra of self-similar spectral measures has been extensively studied by using Hadamard triples, see, e.g., 
\cite{AnDongHe2022, Dai2016,  DaiHeLai2013, DHS2009, FuHe2017, FuHeWen2018, HeTangWu2019} and references therein.

 However, the existence of a Hadamard triple is not necessary for a self-similar measure to be spectral.  Recently, motivated by the study of the product-form self-similar tiles (see, e.g., \cite{LLR2017}), the first-named author and Lai \cite{AnLai2023} introduced a broader class of spectral self-similar  measures generated by \emph{product-form} Hadamard triples:

\begin{defi}[\!\!\cite{AnLai2023}]\label{defi1}
Let $N\geq 2$ be an integer and $\mathcal{A}=\{a_s: s=0,1,\ldots,n-1\}$ be a subset of integers. For each $s$, let $\mathcal{B}_s$ be another finite subset of integers. We say that $\mathcal{D}$ is a product-form digit set generated by  Hadamard triples $(N, \mathcal{A}, \mathcal{L}_1)$ and $(N, \mathcal{B}_s, \mathcal{L}_2)$ if there exists $r\geq1$ such that
\begin{equation}\label{eq1}
\mathcal{D}=\mathcal{D}_{r}=\bigcup_{s=0}^{n-1} \left(a_s+N^r\mathcal{B}_s\right)
\end{equation}
and we have the following conditions for $\mathcal{A}$, $\mathcal{B}_s$, $\mathcal{L}_{1}$ and $\mathcal{L}_{2}$:
	\begin{enumerate}
\item[(i)] $(N, \mathcal{A},\mathcal{L}_1)$, $(N, \mathcal{B}_s,\mathcal{L}_2)$ are  Hadamard triples for all $s=0,1,\ldots,n-1$;

\item[(ii)] $(N, \mathcal{A}\oplus \mathcal{B}_s, \mathcal{L}_1\oplus \mathcal{L}_2)$ are  Hadamard triples for all $s=0,1,\ldots,n-1$.
	\end{enumerate}
We will call $(N, \mathcal{D}, \mathcal{L}_{1}\oplus \mathcal{L}_{2})$ a product-form Hadamard triple if there exist $\mathcal{A}$ and $\mathcal{B}_s$ such that $\mathcal{D}$ is written as in (\ref{eq1}) and {\rm (i), (ii)} hold.
\end{defi}
Here and in what follows, for two finite sets $A$ and $B$, we write $$A\oplus B=\left\{a+b: a\in A, b\in B\right\},$$ where all elements $a+b$ are distinct, so that $\#(A\oplus B)=(\#A)(\#B)$.  This concept can also be defined for more summands. 

In the same paper \cite{AnLai2023}, the authors  proved that the  associated self-similar measure is  a spectral measure by  demonstrating the existence of a spectrum.   
This naturally leads to the following two fundamental problems: 
 
\begin{problem}\label{prom.1}
Given a spectral self-similar  measure $\mu_{N, \D}$ generated by a product-form Hadamard triple,  how can we characterize self-replicating translation sets that are spectra?
\end{problem}

\begin{problem}\label{prom.d.s}
    Does the dual spectral set conjecture hold for self-replicating translation sets?
\end{problem}
A set ${\mathcal J}$  is called  a {\it self-replicating translation set} with respect to $(N, \D)$ if ${\mathcal J}= {\mathcal D}+N{\mathcal J}$. Addressing Problems  \ref{prom.1}  and \ref{prom.d.s} will lead  to several novel applications.

\begin{itemize}
\item 
The case \(r=0\) with all $\B_s$ identical yields a standard Hadamard triple $(N, \D, \mathcal{L}_1\oplus \mathcal{L}_2)$, which has been  studied  by Strichartz \cite{Str1998}, 
{\L}aba and Wang \cite{LW2002}, Dutkay and   Jorgensen\cite{DJ2006,DJ2007,DJ2009}, Dutkay and Lai  \cite{DL2017}, and Dutkay, Haussermann and Lai \cite{DHL2019}.  As their applications, Strichartz \cite{Str2000,Str2006} and Dutkay, Han and Sun \cite{DHS2014} pioneered the study of the convergence and divergence of mock Fourier series associated to different self-replicating spectra, respectively.   Surprisingly, the choice among different self-replicating spectra  dramatically affects the convergence behavior. Problem \ref{prom.1} suggests that  the theory of Fourier series can also be studied in such a relaxed form. We anticipate such a study in the future. 

\smallskip

\item 
  Problem \ref{prom.d.s} will shed some light on the dual spectral set conjecture on self-similar tiles.  Bandt \cite{B91} showed that when \(\#D=N\) and \(T(N,\mathcal{D})\) has non-empty interior, it is a  {\it translational tile}. Such a set \(T(N,\mathcal{D})\)  is known as {\it self-similar tile}   and \(\D\) is referred to as  a {\it tile digit set} with respect to \(N\).  Recently, Li and Rao \cite{LR2025} completely characterized the tile digit sets by employing skew-product-form digit sets, which were originally  introduced  by the first-named author with Lau \cite{AL19} and with Lai \cite{AnLai2023}. A set $\D$ is called a {\it skew-product-form digit set} with respect to $N$ if
there exist $\A=\{a_s\}_{s=0}^{n-1} \subset \mathbb{Z} \text{ and } \B_0, \B_1, \ldots, \B_{n-1} \subset \mathbb{Z}$  such that
$$\D_r = \bigcup_{s=0}^{n-1} (a_s + N^r \B_s)$$
and $\A \oplus \B_s$ is a complete residue set modulo $N$ for all $s = 0, 1, \ldots, n-1$. 
 Lau and Rao \cite{LR03} showed that every self-similar tile in $\R$ has a unique periodic and self-replicating tiling set. As we shall see in Theorem \ref{thm.dua.} below, the self-replicating spectrum-tiling connection seems to be equally compelling.
\end{itemize}

\subsection{Main results}\label{sec.1.main}
The first objective of this paper is to study Problem \ref{prom.1} in the case $r\geq 1$. According to the argument of \cite{AnLai2023},  the digit set with product-form Hadamard triples can be reduced to the case $r=1$ with some higher power of $N$. As the spectrality of $\mu_{N,\mathcal{D}}$ is invariant under translations and scalings of $\mathcal{D}$, we assume throughout this paper that
\begin{align}\label{eq6}
\begin{split}
    &\mathcal{D}=\D_1=\bigcup_{s=0}^{n-1} \left(a_s+N\mathcal{B}_s\right) \quad \text{where $0\in\mathcal{D}$
  and $\gcd(\mathcal{D}_{0})=1$, $\mathcal{D}_{0}=\bigcup_{s=0}^{n-1}\left(a_s+\mathcal{B}_{s}\right)$.}\\
    &\text{$\widetilde{\mathcal L}=\mathcal{L}_1\oplus\frac{\mathcal{L}_2}{N},\qquad\mathcal{L}=\mathcal{L}_1\oplus\mathcal{L}_2$ \quad with \quad $0\in \mathcal{L}_1\cap\mathcal{L}_2$.}   
\end{split}
    \end{align}
As in \cite{AnLai2023}, under the assumption of product-form Hadamard triples, the set  $\widetilde{\mathcal L}$ is a spectrum of the discrete measure $\delta_{\frac{\D}{N}}$. Moreover, for each integer $p\geq1$, the set 
$$\widetilde{\bf L}_p=\widetilde{\mathcal L}+ N\widetilde{\mathcal L}+\cdots+ N^{p-1}\widetilde{\mathcal L}$$ is a spectrum of the convolution of the first $p$ discrete measures $\mu_p=\delta_{\frac{\D}{N}}\ast\cdots\ast\delta_{\frac{\D}{N^p}}$. Here and below, for a finite set $E,$ the symbol  $\delta_{E}$ denotes the discrete measure $\delta_{E}:=\frac{1}{\# E}\sum_{e\in E}\delta_{e}$, where $\delta_e$ is the Dirac measure concentrated at $e$. Then the set $$\Lambda(N, \widetilde{\mathcal L})=\bigcup_{p=1}^\infty \widetilde{\bf L}_p$$ forms an orthonormal set of $\mu_{N,\D}$, where $\mu_{N,\mathcal{D}}$ is the weak limit of $\mu_p$. That is, $E(\Lambda(N, \widetilde{\mathcal L}))$ forms an orthonormal set in $L^2(\mu_{N,\mathcal{D}})$.

This paper investigates Problem \ref{prom.1} by analyzing the following problem: 
\begin{problem}\label{prom2}
With the above notation,  does the set $\La\left(N,\widetilde{\mathcal L}\right)$ form a spectrum of $\mu_{N,\mathcal{D}}$?
  
   In general, let $\Lambda$ be a self-replicating translation set with respect to $(N,\widetilde{{\mathcal L}})$, that is,  $\Lambda=\widetilde{{\mathcal L}}\oplus N\Lambda$. 
   What are the necessary and sufficient conditions for $\La$ to be a spectrum of $\mu_{N,\mathcal{D}}$?
   \end{problem}
  

The problem is  far more challenging in this regime because \(\widetilde{\mathcal L}\) is no longer contained in $\Z$ and the exponential set \(E(\Z)\) is not complete in $L^2(\mu_{N, {\mathcal D}})$. A further complication arises from the fact that such a pair $(\D, \widetilde{\mathcal L})$  cannot be used to construct  a Ruelle transfer operator as in \eqref{R-1}. To overcome these difficulties, this paper develops a new approach in Section \ref{sec.2}.

In order to state our main results, we need to introduce the necessary notation. 
Consider the IFS $\left\{\phi_{\ell}(\xi)=N^{-1}(\xi+\ell): \ell \in \mathcal{L}\right\}$.  
 Let  $[-r,r]$ be a  closed interval  that is invariant under the maps
$\phi_{\ell}(\xi)$ and contains both the self-similar set $T(N,\mathcal{L})$ and the finite set $\frac{\mathcal L}{N-1}$.  A real number  $\xi$ is said to be  a {\it  cycle point} for ${\mathcal L}$ if there exist  $m\in\N$,
$\ell_1,\dots,\ell_m\in\mathcal{L}$ such that $\xi$ is the fixed point of the composite function $\phi_{\ell_m}\circ\cdots\circ\phi_{\ell_1}$. If, in addition, $\xi\in\mathbb{Z}$, then $\xi$ is called an {\it integral cycle point}.

Suppose that the self-replicating translation set $\Lambda=\widetilde{{\mathcal L}}\oplus N\Lambda$ satisfies  that $\Lambda \subset \frac{\mathbb{Z}}{N}$. Then it is straightforward to verify that
 \begin{equation}\label{eq.Lambda}
  \Lambda=\frac{\mathcal{L}_2}{N}\oplus \mathcal{L}_1\oplus N\La=:\frac{\mathcal{L}_2}{N}\oplus\Gamma,
 \end{equation} 
where $\Gamma=\mathcal{L}_1\oplus N\La\subset\Z$.  Here, $\La$ is an orthonormal set of $\mu_{N,\mathcal{D}}$  (see Lemma \ref{prop.rep1}). Denote
$$\mathcal{O}:=\left\{\xi\in [-r,r]: \frac{1}{n}\sum_{s=0}^{n-1}\left|M_{\mathcal{B}_{s}}(\xi)\right|^{2}>0\right\} ,$$ 
and
 $$\mathcal{W}(\widehat{\mu}_{N,\mathcal{D}}):=\left\{\xi\in \mathcal{O}: \widehat{\mu}_{N,\mathcal{D}}(\xi+\gamma)=0, \forall\gamma\in \Gamma\right\},$$
where $M_{\mathcal{B}_{s}}(\xi)= \frac{1}{\#\mathcal{B}_{s}}\sum_{b\in\mathcal{B}_{s}}e^{-2\pi ib\xi}$
and $\widehat{\mu}_{N,\mathcal{D}}$ denotes the Fourier transform of ${\mu}_{N,\mathcal{D}}$. 

Our first main result is stated as Theorem \ref{th2}, which essentially solves Problems \ref{prom.1} and \ref{prom2}.
\begin{theo}\label{th2}
With the above notation, the following statements are equivalent:
	\begin{enumerate}
\item[(i)] $\La$ is not a spectrum of $\mu_{N,\mathcal{D}}$;
\item[(ii)] There exists an integral cycle point $\xi$ for ${\mathcal L}$ such that $-\xi\not\in\Gamma$;
\item[(iii)] $\mathcal{W}(\widehat{\mu}_{N,\mathcal{D}})\neq\emptyset$.
\end{enumerate}
\end{theo}


Let $\mu_{N,\mathcal{A}\oplus\mathcal{B}_{s}}$ be the self-similar meausure generated by the IFS $\left\{\phi_{d}(x)=N^{-1}(x+d): d\in \mathcal{A}\oplus\mathcal{B}_{s}\right\} $, which, by Definition \ref{defi1} {\rm(ii)}, arises from the Hadamard triple  $(N, \A\oplus\B_s, {\mathcal L})$. We next show that, under a mild condition, the equivalence of the spectral pairs  $(\mu_{N,\mathcal{D}},\Lambda)$ and $\left(\mu_{N,\mathcal{A}\oplus\mathcal{B}_{s}},\Gamma\right)$.

\begin{theo}\label{th3}
Let $(N, \D, {\mathcal L})$ be a product-form Hadamard triple and $\Lambda=\frac{\mathcal{L}_2}{N}\oplus \Gamma$ with  $\Gamma=\mathcal{L}\oplus N\Gamma\subset\Z$. Then the following two statements hold:
\begin{enumerate}
\item[(i)]  If there exists $s\in\{0,\ldots,n-1\}$ such that
$\left(\mu_{N,\mathcal{A}\oplus\mathcal{B}_{s}},\Gamma\right)$ is a spectral pair, 
so is $\left(\mu_{N,\mathcal{D}},\Lambda\right)$.
\item[(ii)]  If $(\mu_{N,\mathcal{D}},\Lambda)$ is a spectral pair,  so is $\left(\mu_{N,\mathcal{A}\oplus\mathcal{B}_{s}},\Gamma\right)$ for every $s\in\{0,\ldots,n-1\}$ with
$\gcd(\mathcal{A}\oplus\mathcal{B}_{s})=1$.
\end{enumerate}
\end{theo}
  In particular, taking  $\Gamma=\Lambda(N,\mathcal{L})$ in  \eqref{eq.Lambda},  we obtain  
  $\Lambda=\frac{\mathcal{L}_2}{N}\oplus \Lambda(N,\mathcal{L})=\Lambda(N, \widetilde{\mathcal L}).$  Therefore,  Theorem \ref{th3} immediately answers Problem \ref{prom2} for the special set $\Lambda(N, \widetilde{\mathcal L})$.
  
 \begin{coro}\label{th1}
Let $(N, \D, {\mathcal L})$ be a product-form Hadamard triple. If there exists $s\in\{0,\ldots,n-1\}$ such that
$\left(\mu_{N, \A\oplus\B_s},  \Lambda(N,\mathcal{L})\right)$ is a spectral pair, then so is 
 $\left(\mu_{N, \D}, \Lambda(N, \widetilde{\mathcal L})\right)$. Conversely, if $\left(\mu_{N, \D}, \Lambda(N, \widetilde{\mathcal L})\right)$  is a spectral pair, then so is  $\left(\mu_{N, \A\oplus\B_s}, \Lambda(N,\mathcal{L})\right)$ for every $s\in\{0,\ldots,n-1\}$ with
$\gcd(\mathcal{A}\oplus\mathcal{B}_{s})=1$.  
\end{coro}  

Theorem \ref{th2} gives a criterion to determine whether $\La(N, \widetilde{\mathcal L})$ is a spectrum of $\mu_{N, \D}$. Moreover, according to Lemma \ref{lem.ha} {\rm(i)},  there are two sets $\widehat{\mathcal L}_1, \widehat{\mathcal L}_2$ such that  $(N, \D, \widehat{\mathcal L}_1\oplus\widehat{\mathcal L}_2)$ still forms a product-form Hadamard triple and 0 is the unique integral cycle point for $\widehat{\mathcal L}_1\oplus\widehat{\mathcal L}_2$. This allows us to construct a concrete spectrum $\La(N, \widetilde{\mathcal L})$  for the measure $\mu_{N, \D}$. 

\begin{theo}\label{th5}
Let $(N, \D, {\mathcal L})$ be a product-form Hadamard triple. Then there exists $\widetilde{\mathcal L}$ such that  $\La(N, \widetilde{\mathcal L})$ is a spectrum of the self-similar measure $\mu_{N, \D}$.
\end{theo}

It appears to be a strong link between product-form Hadamard triples,
self-similar tile digit sets and the dual spectral set conjecture. Relying on a 
conjecture proposed by Coven and Meyerowitz \cite{CM1999}, if $\A$ is a self-similar tile digit set, then $\A$ is a part of  a product-form Hadamard triple. A
finite set $\A\subset\Z$ is called a {\it complementing set}$\pmod{M}$ if there exists a
$\B$ such that $\A\oplus\B$ is a complete residue system$\pmod{M}$, where $M$ is an integer. The set $\A$ is called a \textit{spectral set} if the discrete measure $\delta_\A:=\frac{1}{\#\A}\sum_{a\in\A}\delta_a$ is spectral. Finally, we solve Problem \ref{prom.d.s} for a self-replicating translation set. 
\begin{theo}\label{thm.dua.}
  Let $\D\subset\Z$ be a finite digit set with cardinality $N$ and $\Lambda\subset\Z$ be a  self-replicating translation set with respect to $(N,\D)$.
\begin{enumerate}
\item[(i)] If $\Lambda$ is a spectrum, then it is a tiling set.
\item[(ii)] Suppose that any complementing set is a spectral set. If $\Lambda$ is a tiling set, then it is a spectrum. 
\end{enumerate}
\end{theo}
 The rest of the paper is organized as follows. In Section \ref{sec.2}, we first introduce the necessary notation and then present two criteria to determine whether the set $\Lambda$ in \eqref{eq.Lambda} forms an orthonormal set (resp. a spectrum) of the self-similar measure $\mu_{N, \D}$ (Theorems \ref{fomula1} and \ref{fomula2}).  Theorem \ref{fomula2} relies on the Ruelle transfer operator constructed in Lemma \ref{lem1}. This operator and the above two criteria serve as key ingredients for proving  Theorem \ref{th2}. In  Section \ref{sec.3}, we  prove Theorems \ref{th2}, \ref{th3} and \ref{th5}, including  a  complete characterization of the set of integral cycle points for $\mathcal L$. In Section \ref{sec.ex}, we  provide examples to illustrate Theorem \ref{th2} and give the proof of Theorem \ref{thm.dua.}. 
   
\section{Preliminaries}\label{sec.2}
In this section, we first introduce the necessary knowledge for spectral measures and Hadamard triples, and then provide a criterion (Theorem \ref{fomula1}) under which an orthonormal set $\Lambda$ is a spectrum of the measure $\mu_{N, \D}$, which completes the characterization due to \cite{AnLai2023}. Finally, we define a new Ruelle operator that yields a new criterion (Theorem \ref{fomula2}) and serves as a key tool in the proof of the main theorem (Theorem \ref{th2}). 


\subsection{Notation and known results}
Throughout the paper,  we use $e(x)$ to denote $e^{2\pi i x}$, and the Fourier transform of a measure $\mu$ is defined by
 \begin{equation*}
\widehat{\mu}(\xi)=\int e(-\xi x)\dmu(x), \qquad \xi\in\R.
 \end{equation*}
The criterion below provides a universal test to determine whether a set $\Lambda \subset \R$ is an orthonormal set (resp. a spectrum) of a measure $\mu$. For $\mu$ and a discrete set $\Lambda\subset\R$, we always write 
$$Q_{\Lambda}(\xi)=\sum_{\lambda \in \Lambda}\left|\widehat{\mu}(\xi+\lambda)\right|^{2}, \qquad  \xi\in \R.$$
\begin{theo}[\!\cite{JP1998}]\label{cth2}
Let $\mu$ be a Borel probability measure on $\R$ and $\Lambda\subset\R$ be a countable subset. Then
\begin{enumerate}
\item[(i)] $\Lambda$ is an orthonormal set of $\mu$ if and only if $Q_{\Lambda}(\xi)\leq1$ for all $\xi\in\R$. In this case, $Q_{\Lambda}(\xi)$ is an entire function on $\C$.
\item[(ii)] $\Lambda$ is a spectrum of $\mu$ if and only if $Q_{\Lambda}(\xi)\equiv1$ on $\R$.
\end{enumerate}
\end{theo}
For a finite set  ${\mathcal D}\subset \R$, we write
$$
M_{\D}(\xi) = \widehat{\delta_{\D}}(\xi) = \frac{1}{\#\D} \sum_{d\in \D} e(-d\xi).
$$
We record the properties of Hadamard triples, whose proof is standard and therefore omitted.
\begin{lem}[\!\cite{LW2002}]\label{lem.ha}
Let $1<N\in\N$  and $\D, {\mathcal L}\subset\Z$ such that $(N,\D,{\mathcal L})$ forms a Hadamard triple. 
\begin{enumerate}
\item[(i)] If $\widehat{\mathcal L}\subset\Z$ satisfies $\widehat{\mathcal L}\equiv{\mathcal L}\pmod{N}$, then $(N,\D,\widehat{\mathcal L})$ forms a Hadamard triple.
\item[(ii)] The elements in ${\mathcal L}$ are in distinct residue modulo $N$.
\item[(iii)] $\frac{1}{N}{\mathcal L}$ is a spectrum of $\delta_{\D}$, i.e.
    $$\sum_{\ell\in{\mathcal L}} \left|M_{\frac{\D}{N}}(\xi+\ell)\right|^2 = 1, \qquad \forall \ \xi\in{\mathbb R}.$$
\end{enumerate}
\end{lem}

\subsection{The criteria for the orthogonality and completeness}

In this subsection, we present criteria characterizing when a set of the form $\Lambda = \frac{{\mathcal L}_2}{N} \oplus \Gamma$ in \eqref{eq.Lambda} forms an orthonormal set (resp. a spectrum) of the self-similar measure $\mu_{N, \D}$. The first-named author and Lai \cite{AnLai2023} provided a sufficient condition for constructing such spectra. The following theorem complements  the characterization by showing that this condition is also necessary.

\begin{theo}\label{fomula1}
Let $\mu_{N,\D}$ be a self-similar measure generated by the product-form Hadamard triple $(N, \D,\mathcal{L})$ given in \eqref{eq6}. Suppose that $\Gamma\subset\Z$. Then $\Lambda=\frac{{\mathcal L}_2}{N}\oplus\Gamma$ is a spectrum of $\mu_{N,\mathcal{D}}$ if and only if $$Q_{\Gamma}(\xi)= \frac{1}{n}\sum_{s=0}^{n-1}\left|M_{\mathcal{B}_{s}}(\xi)\right|^{2}, \qquad \forall\ \xi\in\R.$$
\end{theo}

Before proving Theorem \ref{fomula1}, we first establish several inequalities for some orthonormal sets. The self-similar measure $ \mu_{N,\D}$ satisfies the following infinite product formula
$$
\widehat{\mu}_{N,\D}(\xi) = \prod_{j=1}^{\infty} M_{\D} \left(\frac{\xi}{N^j}\right).
$$
We  can decompose $\mu_{N,\D}$ as a finite convolution $\mu = \mu_p \ast \mu_{>p}$, where
$$
\mu_p = \delta_{\frac{\D}{N}}\ast\cdots\ast \delta_{\frac{\D}{N^p}}
$$
and $\mu_{>p}$ is  the convolution of the remaining factors.

\bigskip

\begin{lem}\label{cprop1}
Suppose that $\Gamma\subset\Z$. Then $\Gamma$ is an orthonormal set of $\mu_p$ if and only if
\begin{equation}\label{eq10}
\sum_{\gamma\in\Gamma}|\widehat{\mu}_p(\xi+\gamma)|^2\leq \frac{1}{n}\sum_{s=0}^{n-1}\left|M_{\mathcal{B}_{s}}(\xi)\right|^{2}, \qquad \forall\ \xi\in\mathbb{R}.
\end{equation}
\end{lem}
\begin{proof}
$(\Leftarrow)$
Taking $\xi=-\gamma_0\in-\Gamma\subset\Z$ in \eqref{eq10},  we obtain
\begin{eqnarray*}
\sum_{\gamma\in\Gamma}|\widehat{\mu}_p(\gamma-\gamma_0)|^2
=1+\sum_{\gamma\in\Gamma\setminus\{\gamma_0\}}|\widehat{\mu}_p(\gamma-\gamma_0)|^2
\le \frac{1}{n}\sum_{s=0}^{n-1}\left|M_{\mathcal{B}_{s}}(-\gamma_0)\right|^{2}=1.
\end{eqnarray*}
Thus, for any $\gamma\neq\gamma_0$,
$$\widehat{\mu}_p(\gamma-\gamma_0)=0.$$
Since the choice of $\gamma_0\in\Gamma$ is arbitrary, the set $\Gamma$ forms an orthonormal set of $\mu_p$.

\bigskip

$(\Rightarrow)$
Let ${\bf D}_p=\sum_{k=1}^p\frac{\D}{N^k}$ and $m=\# \mathcal{B}_s$ $(s=0,1,\ldots,n-1)$. For fixed $\xi\in\R$ and $\gamma\in\Gamma$, let 
$${{\bf w}_{\xi}}=\frac{1}{\sqrt{n^pm^{p-1}}}\left(e\left(-\frac{(a_s+d)\xi}{N}\right)M_{{\mathcal B}_s}(-\xi)\right)_{a_s+d\in {\A+{\bf D}_{p-1}}}$$
and
$$
{\bf v}_{\gamma}=\frac{1}{\sqrt{n^pm^{p-1}}}\left(e\left(\frac{(a_s+d)\gamma}{N}\right)\right)_{a_s+d\in {\A+{\bf D}_{p-1}}}
$$
be two vectors 
in $\C^{n^p m^{p-1}}$. Since $\D=\bigcup_{s=0}^{n-1}\left(a_s+N\B_s\right)$ and
$\mu_p=\delta_{\frac{\D}{N}}\ast\delta_{\frac{\D}{N^2}}\ast\cdots\ast\delta_{\frac{\D}{N^p}}=\delta_{\frac{\D+{\bf D}_{p-1}}{N}},$ it follows that 
\begin{eqnarray}\label{eq-inner}
\sum_{\gamma\in\Gamma}|\widehat{\mu}_p(\xi+\gamma)|^2&=&\sum_{\gamma\in\Gamma}\left|\frac{1}{(nm)^p}\sum_{a_s\in \A}\sum_{b\in\B_s}\sum_{d\in{\bf D}_{p-1}}e\left(\frac{(a_s+Nb+d)}{N}(\xi+\gamma)\right)\right|^2\nonumber\\
&=&\sum_{\gamma\in\Gamma}\left|\frac{1}{n^pm^{p-1}}\sum_{a_s\in \A}\sum_{d\in{\bf D}_{p-1}}e\left(\frac{(a_s+d)\gamma}{N}\right)\cdot e\left(\frac{(a_s+d)\xi}{N}\right)M_{B_s}(\xi)\right|^2\nonumber\\
&=&\sum_{\gamma\in\Gamma}\left|\langle {\bf v}_{\gamma}, {\bf w}_{\xi} \rangle\right|^2,
\end{eqnarray}
where the second equality holds  because $\mathcal{B}_s\subset \Z$ for each $s = 0,1,\ldots,n-1$  and $\Gamma \subset\mathbb{Z}$ .

Note that for any integer $t$, we have
$$M_{\frac{\D}{N}}(t)=\frac1n\sum_{s=0}^{n-1}e^{-2\pi i \frac{a_s t}{N}}M_{B_s}(t)=M_{\frac{\A}{N}}(t),$$
 which yields that
\begin{eqnarray*}
{\widehat{\mu}_p}(t)=M_{\frac{\D}{N}}(t)\cdot M_{\frac{{\bf D}_{p-1}}{N}}(t)=M_{\frac{\A}{N}}(t)\cdot M_{\frac{{\bf D}_{p-1}}{N}}(t)
=\widehat{{\delta}}_{\frac{\A+{\bf D}_{p-1}}{N}}(t).
\end{eqnarray*}
Thus, if $\Gamma \subset \mathbb{Z}$, then $\Gamma$ is an orthonormal set of $\mu_p$ if and only if it is an orthonormal set of the discrete measure
$\delta_{\frac{\mathcal{A}+\mathbf{D}_{p-1}}{N}}$. So the following set of vectors 
$$\left\{{\bf v}_{\gamma}=\frac{1}{\sqrt{n^pm^{p-1}}}\left(e\left(\frac{(a_s+d)\gamma}{N}\right)\right)_{a_s+d\in {\A+{\bf D}_{p-1}}}: \gamma\in\Gamma\right\}$$
 forms an orthonormal system in $\C^{n^p m^{p-1}}$. 
 Applying Bessel's inequality to \eqref{eq-inner} yields the following inequality
\begin{eqnarray*}
\sum_{\gamma\in\Gamma}|\widehat{\mu}_p(\xi+\gamma)|^2=\sum_{\gamma\in\Gamma}\left|\langle {\bf v}_{\gamma}, {\bf w}_{\xi} \rangle\right|^2\le\|{\bf w}_{\xi}\|^2=\frac{1}{n}\sum_{s=0}^{n-1}\left|M_{\mathcal{B}_{s}}(\xi)\right|^{2},
\end{eqnarray*} where $\|\cdot\|$ denotes the matrix $2$-norm of $\C^{n^p m^{p-1}}$. 
This completes the proof.
\end{proof}

\bigskip

\begin{lem}\label{cprop2}
Suppose that $\Gamma\subset\Z$. Then $\Lambda=\frac{{\mathcal L_2}}{N}\oplus \Gamma$ is an orthonormal set of $\mu_{N, \D}$ if and only if
\begin{equation*}
Q_{\Gamma}(\xi)\leq \frac{1}{n}\sum_{s=0}^{n-1}\left|M_{\mathcal{B}_{s}}(\xi)\right|^{2}, \qquad \forall\ \xi\in\mathbb{R}.
\end{equation*}
\end{lem}

\begin{proof} $(\Leftarrow)$  By interchanging the order of  summation,
$$
Q_{\La}(\xi)=\sum_{\ell_2\in{\mathcal L}_2}\sum_{\gamma\in\Gamma} \left|\widehat{\mu}_{N,\D}\left(\xi+\frac1N\ell_2+\gamma\right)\right|^2 \le \frac1n\sum_{s=0}^{n-1} \sum_{\ell_2\in{\mathcal L}_2}\left|M_{{\mathcal B}_s}\left(\xi+\frac1N\ell_2\right)\right|^2.
$$
As $\frac{1}{N}{\mathcal L}_2$ is a spectrum of $\delta_{{\mathcal B}_s}$ for all $s=0,\ldots ,n-1$, it follows from Lemma \ref{lem.ha} that the above sum on the right-hand side is equal to $1$.   This shows that $\La$ is an orthonormal set of $\mu_{N, \D}$.

\medskip

$(\Rightarrow)$ Suppose that $\Lambda$ is an orthonormal set of $\mu_{N, \D}$. Since $0\in\mathcal{L}_2$, the set $\Gamma$ is also an orthonormal set of $\mu_{N, \D}$. Write $$\Gamma=\bigcup_{k=1}^{\infty}\Gamma_k\quad \text{with} \quad \Gamma_k:=\Gamma\cap [-k, k].$$
 Therefore, for each $k\geq1$, there is an integer $p_k$  such that $\Gamma_k$ is an orthonormal set of $\mu_{p_k}$. By Lemma \ref{cprop1}, 
\begin{equation*}
\sum_{\gamma\in\Gamma_k}|\widehat{\mu}_{p_k}(\xi+\gamma)|^2
\leq \frac{1}{n}\sum_{s=0}^{n-1}\left|M_{\mathcal{B}_{s}}(\xi)\right|^{2},
\qquad \forall\ \xi\in\mathbb{R}.
\end{equation*}
Since $|\widehat{\mu}_{N, \D}(\xi)|\le |\widehat{\mu}_{p_k}(\xi)|$ for each $k\geq1$ and for every $\xi\in\R$,  we have
\begin{eqnarray*}
Q_{\Gamma}(\xi)=\lim_{k\to\infty}\sum_{\gamma\in\Gamma_k}|\widehat{\mu}_{N, \D}(\xi+\gamma)|^2
\le\lim_{k\to\infty}\sum_{\gamma\in\Gamma_k}|\widehat{\mu}_{p_k}(\xi+\gamma)|^2
\leq \frac{1}{n}\sum_{s=0}^{n-1}\left|M_{\mathcal{B}_{s}}(\xi)\right|^{2}.
\end{eqnarray*}
The proof of Lemma \ref{cprop2} is complete. 
\end{proof}

\bigskip

Now we are ready to prove Theorem \ref{fomula1}.

\bigskip

\noindent{\bf Proof of Theorem \ref{fomula1}.}  The sufficiency has been proved in \cite{AnLai2023}, we give it here for the sake of completeness. By interchanging summation,
$$
Q_{\Lambda}(\xi)=\sum_{\ell_2\in{\mathcal L}_2}\sum_{\gamma\in\Gamma} \left|\widehat{\mu}_{N,\D}\left(\xi+\frac1N\ell_2+\gamma\right)\right|^2 = \frac1n\sum_{s=0}^{n-1} \sum_{\ell_2\in{\mathcal L}_2}\left|M_{{\mathcal B}_s}\left(\xi+\frac1N\ell_2\right)\right|^2.
$$
As $\frac{1}{N}{\mathcal L}_2$ is a spectrum of $\delta_{{\mathcal B}_s}$ for all $s=0,\ldots ,n-1$, the sum on the right-hand side above is equal to $1$.   Hence,  Theorem \ref{cth2}  \rm{(ii)} implies that $\Lambda$ is a spectrum of $\mu_{N, \D}$.

Next,  we  prove the necessity by contradiction. Suppose that $\Lambda$ is a spectrum of $\mu_{N,\mathcal{D}}$, yet there exists some  $\xi_0\in \R$ such that
$$Q_{\Gamma}(\xi_0)\neq \frac{1}{n}\sum_{s=0}^{n-1}\left|M_{\mathcal{B}_{s}}(\xi_0)\right|^{2}.$$
According to Lemma \ref{cprop2}, the orthogonality of $\La$ implies that
\begin{equation*}
Q_{\Gamma}(\xi)\le\frac{1}{n}\sum_{s=0}^{n-1}\left|M_{\mathcal{B}_{s}}(\xi)\right|^{2}, \qquad \forall\ \xi\in\R.
\end{equation*}
In particular,
\begin{equation*}
Q_{\Gamma}(\xi_0)< \frac{1}{n}\sum_{s=0}^{n-1}\left|M_{\mathcal{B}_{s}}(\xi_0)\right|^{2}.
\end{equation*}
Thus, 
 \begin{eqnarray*}
 Q_{\La}(\xi_0)&=&\sum_{\ell_2\in\mathcal{L}_{2}}\sum_{\gamma\in\Gamma}\left|\widehat{\mu}_{N, \D}\left(\xi_0+\frac{\ell_2}{N}+\gamma\right)\right|^2\\
 &=&\sum_{\ell_2\in\mathcal{L}_{2}}Q_{\Gamma}\left(\xi_0+\frac{\ell_{2}}{N}\right)\\
&=&Q_{\Gamma}\left(\xi_0\right)+\sum_{\ell_2\in\mathcal{L}_{2}\setminus\{0\}}Q_{\Gamma}\left(\xi_0+\frac{\ell_{2}}{N}\right)\\
 &<&\frac{1}{n}\sum_{s=0}^{n-1}\left|M_{\mathcal{B}_{s}}(\xi_0)\right|^{2}
 +\sum_{\ell_2\in\mathcal{L}_{2}\setminus\{0\}}\frac{1}{n}
 \sum_{s=0}^{n-1}\left|M_{\mathcal{B}_{s}}
 \left(\xi_0+\frac{\ell_{2}}{N}\right)\right|^{2}\\
 &=& \frac1n\sum_{s=0}^{n-1}\sum_{\ell_2\in\mathcal{L}_{2}}\left|M_{\mathcal{B}_{s}}
 \left(\xi_0+\frac{\ell_{2}}{N}\right)\right|^{2}\\
&=& 1,
 \end{eqnarray*}
where the last equality is true because $\frac{1}{N}{\mathcal L}_2$ is a spectrum of $\delta_{{\mathcal B}_s}$ for all $s=0,...,n-1$. By Theorem \ref{cth2} {\rm(ii)}, $\La$ is not a spectrum of $\mu_{N,\mathcal{D}}$, which yields a contradiction. Hence
$$Q_{\Gamma}(\xi)= \frac{1}{n}\sum_{s=0}^{n-1}\left|M_{\mathcal{B}_{s}}(\xi)\right|^{2}, \qquad \forall\ \xi\in\R.$$
We have completed the proof.
\qed

\bigskip

\subsection{Ruelle transfer operator}
Recall that the self-replicating translation set $\Lambda\subset \frac{\mathbb{Z}}{N}$ has the form $\Lambda=\frac{{\mathcal L}_2}{N}\oplus\Gamma$ where 
$$\Gamma={\mathcal L}_1\oplus N\Lambda={\mathcal L}_1\oplus{\mathcal L}_2 \oplus N\Gamma={\mathcal L}\oplus N\Gamma.$$
Then $\Gamma$ is also a self-replicating translation set. 
The previous subsection asserts that such a spectrum of $\mu_{N,\D}$ can be constructed if we can establish the identity $$Q_{\Gamma}(\xi)
=\frac{1}{n}\sum_{s=0}^{n-1}
\left|M_{\mathcal{B}_{s}}(\xi)\right|^{2}, \qquad \forall \xi\in\R.$$
Unfortunately, a common zero for all  $M_{\mathcal{B}_{s}}$ will create an obstacle.  To address this issue, we apply the technique introduced by the first-named author and Lai \cite{AnLai2023} to eliminate these common zeros. For each $s=0,1,\ldots,n-1$, write  $$M_{\mathcal{B}_s}(\xi)=\frac{1}{\#\mathcal{B}_s}P_{\mathcal{B}_{s}}(e(-\xi)) \quad \text{where} \quad P_{\mathcal{B}_s}(x)=\sum_{b\in\mathcal{B}_s}x^{b}\in \mathbb{Z}[x].$$
For each $s$,  if $e(-\theta)$ is a root of $P_{\mathcal{B}_{s}}$, then its minimal  polynomial $F_{\theta}(x)\in \mathbb{Z}[x]$ divides $ P_{\mathcal{B}_{s}}(x)$.
Define the set of common roots
$$Z=\left\{e(-\theta): P_{\mathcal{B}_{s}}(e(-\theta))=0, \forall s=0,1,\ldots,n-1\right\}.$$
It is clear that $Z$ is a finite set.
For each $e(-\theta)\in Z$ and $s\in\{0,1,\ldots,n-1\}$, denote $$k_{\theta,s}=\max\left\{k:F^{k}_{\theta}(x) \mid P_{\mathcal{B}_{s}}(x)\right\},$$ and set
$$k_{\theta}=\min\left\{k_{\theta,s}:  s=0,1,\ldots,n-1\right\}.$$\\
Under the above notation, define the polynomial
$$F(x)=\Pi_{e(-\theta)\in Z}F^{k_{\theta}}_{\theta}(x).$$
Then $F(x)$ divides each $P_{\mathcal{B}_{s}}(x)$, so for each  $s=0,1,\ldots, n-1$, we may write $$P_{\mathcal{B}_{s}}(x)=F(x)\widetilde{P}_{\mathcal{B}_{s}}(x)\quad \text{where}\quad\widetilde{P}_{\mathcal{B}_{s}}(x)\in \mathbb{Z}[x].$$
Define $$\widetilde{M}_{\mathcal{B}_{s}}(\xi)=\frac{1}{\#\mathcal{B}_{s}}\widetilde{P}_{\mathcal{B}_{s}}(e(-\xi)).$$  Since the common zeros have been completely factored out, the following function is strictly positive:
 \begin{equation}\label{eq.geq.0}
  \frac{1}{n}\sum_{s=0}^{n-1}
 \left|\widetilde{M}_{\mathcal{B}_{s}}(\xi)\right|^2>0,\qquad \forall\xi\in\mathbb{C}.
 \end{equation}
Combining this with the definition of $\D$, we  easily obtain
 \begin{equation}\label{eq.tra}
   M_{\mathcal{B}_{s}}(\xi)=F(e(-\xi))\widetilde{M}_{\mathcal{B}_{s}}(\xi)\;
\quad \text{and}\; \quad
M_{\frac{\mathcal{D}}{N}}(\xi)=F(e(-\xi))\widetilde{M}_{\frac{\mathcal{D}}{N}}(\xi),
 \end{equation}
where $F(e(-\xi))$ and $\widetilde{M}_{\mathcal{B}_{s}}(\xi)$ are $1-$periodic functions, and
\begin{equation*}
\widetilde{M}_{\frac{\mathcal{D}}{N}}(\xi)=\frac{1}{n}\sum_{s=0}^{n-1}e\left( -\frac{a_{s}}{N}\xi\right)\widetilde{M}_{\mathcal{B}_{s}}(\xi).
\end{equation*}
 
The following identity, due to the first-named author and Lai \cite{AnLai2023}, will be useful in the subsequent analysis.
\begin{lem}[\!{\cite[Lemma 4.1]{AnLai2023}}]\label{clem1}
 Let $\left(N,\mathcal{D},\mathcal{L}\right)$ be a product-form Hadamard triple, where $\mathcal{D}$ is given as in \eqref{eq6}. For each $s=0,1,\ldots,n-1$, let $P_{s}(\xi)$ be a $1-$periodic function and define
 $$\mathbf{P}(\xi)=\frac{1}{n}\sum_{s=0}^{n-1}e\left(-\frac{a_{s}\xi}{N}\right)P_{s}(\xi).$$
 Then for each $i=0, 1, \ldots, n-1$,
 $$\sum_{\ell\in\mathcal{L}}\left|\mathbf{P}(\xi+\ell)\right|^2
 \left|M_{\mathcal{B}_{i}}(\phi_{\ell}(\xi))\right|^2
 =\frac{1}{n}\sum_{s=0}^{n-1}\left|P_{s}(\xi)\right|^2.$$
\end{lem}

\bigskip

 Defining the following function
\begin{align}\label{eq-Q1}
\widetilde{Q}_{\Gamma}(\xi)
:= \frac{\sum_{\ell\in \mathcal{L}}\left|\widetilde{M}_{\mathcal{D}}\left(\phi_{\ell}(\xi)\right)\right|^2 \cdot Q_{\Gamma}\left(\phi_{\ell}(\xi)\right)}
{\frac{1}{n}\sum_{s=0}^{n-1}
\left|\widetilde{M}_{\mathcal{B}_{s}}(\xi)\right|^2},\qquad \xi\in\R.
\end{align}
It is well-defined by \eqref{eq.geq.0}. In the following, we establish a criterion based on $\widetilde{Q}_{\Gamma}(\xi)$ to determine the orthogonality and completeness of $\Lambda=\frac{{\mathcal L}_2}{N}\oplus \Gamma$. Before doing this, we discuss the relation between  $\widetilde{Q}_{\Gamma}(\xi)$  and  ${Q}_{\Gamma}(\xi)$.

\bigskip

\begin{lem}\label{lem-Q}
With the above notation, suppose  that $\Gamma\subset\Z$ and $\Gamma={\mathcal L}\oplus N\Gamma$. Then for all $\xi\in\R$,
\begin{eqnarray*}
Q_{\Gamma}\left(\xi\right)=\left( \frac{1}{n}\sum_{s=0}^{n-1}
\left| {M}_{\mathcal{B}_{s}}
\left(\xi\right)\right|^2\right)
\widetilde{Q}_{\Gamma}
\left(\xi\right).
\end{eqnarray*}
\end{lem}
\begin{proof}
The equation $\Gamma=\mathcal{L} \oplus N\Gamma$ implies that
\begin{eqnarray*}
Q_{\Gamma}\left(\xi\right)
&=&\sum_{\gamma\in \Gamma}\left|\widehat{\mu}_{N,\mathcal{D}}\left(\xi+\gamma\right)\right|^2\nonumber\\
&=&\sum_{\ell\in{\mathcal L}}\sum_{\gamma\in\Gamma}\left|M_{\frac{\mathcal{D}}{N}}\left(\xi+\ell+N\gamma\right)\right|^2\left|\widehat{\mu}_{N,\mathcal{D}}\left(\frac{\xi+\ell+N\gamma}{N}\right)\right|^2\nonumber\\
&=&\sum_{\ell\in{\mathcal L}}\left|M_{\frac{\mathcal{D}}{N}}\left(\xi+\ell\right)\right|^2Q_{\Gamma}(\phi_{\ell}(\xi))\nonumber\\
&=&\sum_{\ell\in{\mathcal L}}\left|F(e(-\xi-\ell))\right|^2\cdot \left|\widetilde{M}_{\frac{\mathcal{D}}{N}}\left(\xi+\ell\right)\right|^2Q_{\Gamma}(\phi_{\ell}(\xi))\nonumber\\
&=&\left|F(e(-\xi))\right|^2\cdot \sum_{\ell\in{\mathcal L}}\left|\widetilde{M}_{\mathcal{D}}\left(\phi_{\ell}(\xi)\right)\right|^2Q_{\Gamma}(\phi_{\ell}(\xi)),
\end{eqnarray*}
where the last identity follows from the fact that the function $F\left(e(-\xi)\right)$ is $1-$periodic. Combining this with the expression of $\widetilde{Q}_{\Gamma}(\xi)$ in \eqref{eq-Q1}, we obtain 
\begin{eqnarray*}
Q_{\Gamma}\left(\xi\right)
&=&\left|F(e(-\xi))\right|^2\cdot\left( \frac{1}{n}\sum_{s=0}^{n-1}\left|\widetilde{M}_{\mathcal{B}_{s}}(\xi)\right|^2\right) \widetilde{Q}_{\Gamma}(\xi)\\
&=&\left( \frac{1}{n}\sum_{s=0}^{n-1}\left|{M}_{\mathcal{B}_{s}}(\xi)\right|^2\right) \widetilde{Q}_{\Gamma}(\xi).
\end{eqnarray*} The proof is finished. 
\end{proof}

\bigskip

 We next show that the function $\widetilde{Q}_{\Gamma}$  is a fixed point of  some  Ruelle transfer operator  $\mathcal{R}$, which will be introduced below, that is,
\begin{equation*}
\mathcal{R}\left(\widetilde{Q}_{\Gamma}\right)=\widetilde{Q}_{\Gamma}.
\end{equation*} Define a class of functions as
 \begin{equation}\label{eq.varphi}
\mathbf{M}_{\ell}(\xi):
=\left|\widetilde{M}_{\mathcal{D}}\left(\phi_{\ell}(\xi)\right)\right|^2
\frac{\sum_{j=0}^{n-1}\left|M_{\mathcal{B}_{j}}\left(\phi_{\ell}(\xi)\right)\right|^2}
{\sum_{s=0}^{n-1}\left|\widetilde{M}_{\mathcal{B}_{s}}(\xi)\right|^2}, \qquad \ell\in\mathcal{L}.
 \end{equation}

\bigskip

\begin{lem}\label{lem1}
Suppose that $\Gamma\subset\Z$ satisfies $\Gamma={\mathcal L} \oplus N\Gamma$. Thus,  for all $\xi\in\R$, we have 
\begin{equation}\label{eq.rul}
\widetilde{Q}_{\Gamma}(\xi)=\sum_{\ell\in\mathcal{L}}\mathbf{M}_{\ell}(\xi)\widetilde{Q}_{\Gamma}\left(\phi_{\ell}(\xi)\right)
\end{equation}
and
$\sum_{\ell\in\mathcal{L}}\mathbf{M}_{\ell}(\xi)
=1$.
\end{lem}

\begin{proof}
 Using Lemma \ref{lem-Q}, we obtain
$$Q_{\Gamma}\left(\phi_{\ell}(\xi)\right)=\left( \frac{1}{n}\sum_{j=0}^{n-1}
\left| {M}_{\mathcal{B}_{j}}
\left(\phi_{\ell}(\xi)\right)\right|^2\right)
\widetilde{Q}_{\Gamma}
\left(\phi_{\ell}(\xi)\right),\qquad  \xi\in\R,~\ell\in\mathcal{L}.
$$
Substituting it into the numerator on the right-hand side of equation \eqref{eq-Q1}  and interchanging the order of summation yields
\begin{eqnarray*}
 \widetilde{Q}_{\Gamma}(\xi)
 &=&\frac1{\frac{1}{n}\sum_{s=0}^{n-1}\left|\widetilde{M}_{\mathcal{B}_{s}}\left(\xi\right)\right|^2}\sum_{\ell\in\mathcal{L}}
 \left|\widetilde{M}_{\mathcal{D}}
 (\phi_{\ell}(\xi))\right|^2
 \cdot \left(\frac1n\sum_{j=0}^{n-1}\left|M_{\mathcal{B}_{j}}
 \left(\phi_{\ell}(\xi)\right)\right|^2\right)
 \cdot \widetilde{Q}_{\Gamma}\left(\phi_{\ell}(\xi)\right)\\
 &=&\sum_{\ell\in\mathcal{L}}
 \left|\widetilde{M}_{\mathcal{D}}
 (\phi_{\ell}(\xi))\right|^2
 \cdot \left(\frac{\sum_{j=0}^{n-1}\left|M_{\mathcal{B}_{j}}
 \left(\phi_{\ell}(\xi)\right)\right|^2}{\sum_{s=0}^{n-1}\left|\widetilde{M}_{\mathcal{B}_{s}}\left(\xi\right)\right|^2}\right)
 \cdot \widetilde{Q}_{\Gamma}\left(\phi_{\ell}(\xi)\right)\\
 &=& \sum_{\ell\in\mathcal{L}}\mathbf{M}_{\ell}(\xi)
\widetilde{Q}_{\Gamma}
\left(\phi_{\ell}(\xi)\right).
\end{eqnarray*}
This proves the equation \eqref{eq.rul}.

In Lemma \ref{clem1}, taking $$\mathbf{P}(\xi):=\widetilde{M}_{\frac{\mathcal{D}}{N}}(\xi)=\frac{1}{n}\sum_{s=0}^{n-1}e\left( -\frac{a_{s}}{N}\xi\right)\widetilde{M}_{\mathcal{B}_{s}}(\xi).$$
Then
$$\sum_{\ell\in\mathcal{L}}\left|\widetilde{M}_{\mathcal{D}}\left(\phi_{\ell}(\xi)\right)\right|^2
\left|M_{\mathcal{B}_{i}}\left(\phi_{\ell}(\xi)\right)\right|^2=\frac{1}{n}\sum_{s=0}^{n-1}\left|\widetilde{M}_{\mathcal{B}_{s}}(\xi)\right|^2.$$
By Fubini's theorem, we obtain that
\begin{eqnarray*}\label{eq7}
\sum_{\ell\in\mathcal{L}}\mathbf{M}_{\ell}(\xi)
&=&\sum_{i=0}^{n-1}
\frac{\sum_{\ell\in\mathcal{L}}\left|\widetilde{M}_{\mathcal{D}}\left(\phi_{\ell}(\xi)\right)\right|^2
\left|M_{\mathcal{B}_{i}}\left(\phi_{\ell}(\xi)\right)\right|^2}
{\sum_{j=0}^{n-1}\left|\widetilde{M}_{\mathcal{B}_{j}}(\xi)\right|^2}\\
&=&\sum_{i=0}^{n-1}
\frac{\frac{1}{n}\sum_{s=0}^{n-1}\left|\widetilde{M}_{\mathcal{B}_{s}}(\xi)\right|^2}
{\sum_{j=0}^{n-1}\left|\widetilde{M}_{\mathcal{B}_{j}}(\xi)\right|^2} \\
&=&1.
\end{eqnarray*}
We have completed the proof.
\end{proof}

\bigskip

 Lemma \ref{lem1}  tells us that, for each $\xi\in \mathbb{R}$, the nonnegative values  $\left\{\mathbf{M}_{\ell}(\xi)\right\}_{\ell\in\mathcal{L}}$ form a set of probability weights on $\mathcal{L}$. We shall define the associated \textit{Ruelle transfer operator} $\mathcal{R}$ acting on a function $f$ by
\begin{equation*}
\mathcal{R}(f)(\xi):=\sum_{\ell\in\mathcal{L}}\mathbf{M}_{\ell}(\xi)f\left(\phi_{\ell}(\xi)\right).
\end{equation*}
Consequently, $\mathcal{R}(1)=1$ for the constant function $1.$ 
The following provides a criterion on $\widetilde{Q}_{\Gamma}(\xi)$ for determining the orthogonality and completeness of $\Lambda=\frac{{\mathcal L}_2}{N}\oplus\Gamma$.

\bigskip

\begin{theo}\label{fomula2}
Let $\mu_{N,\D}$ be a self-similar measure generated by the product-form Hadamard triple $(N, \D,\mathcal{L})$ given in \eqref{eq6}. Suppose that $\Gamma\subset\Z$ satisfies $\Gamma={\mathcal L} \oplus N\Gamma$. Then
\begin{enumerate}
\item[(i)] $\Lambda=\frac{{\mathcal L}_2}{N}\oplus \Gamma$ is an orthonormal set of $\mu_{N, \D}$ if and only if $\widetilde{Q}_{\Gamma}(\xi)\leq1$ for $\xi\in\R$. In this case $\widetilde{Q}_{\Gamma}(\xi)$ is an entire function on $\C$.
\item[(ii)] $\Lambda=\frac{{\mathcal L}_2}{N}\oplus\Gamma$ is a spectrum of $\mu_{N, \D}$ if and only if $\widetilde{Q}_{\Gamma}(\xi)\equiv1$ on $\R$.
\end{enumerate}
\end{theo}
\begin{proof}
 {\rm(i)}  Assume that $\Lambda=\frac{{\mathcal L}_2}{N}\oplus\Gamma$ is an orthonormal set of $\mu_{N, \D}$. By Lemma \ref{cprop2}, 
$$Q_{\Gamma}\left(\phi_{\ell}(\xi)\right)\le \frac1n\sum_{s=0}^{n-1}\left|M_{\B_s}\left(\phi_{\ell}(\xi)\right)\right|^2$$ holds for all $\xi\in\R$ and $\ell\in{\mathcal L}$, it follows from \eqref{eq.varphi} that 
\begin{eqnarray*}
\widetilde{Q}_{\Gamma}(\xi)
&=& \frac{\sum_{\ell\in \mathcal{L}}\left|\widetilde{M}_{\mathcal{D}}\left(\phi_{\ell}(\xi)\right)\right|^2 \cdot Q_{\Gamma}\left(\phi_{\ell}(\xi)\right)}
{\frac{1}{n}\sum_{s=0}^{n-1}\left|\widetilde{M}_{\mathcal{B}_{s}}(\xi)\right|^2}\\
&\le& \frac{\sum_{\ell\in \mathcal{L}}\left|\widetilde{M}_{\mathcal{D}}\left(\phi_{\ell}(\xi)\right)\right|^2 \cdot \left(\frac1n\sum_{j=0}^{n-1}\left|M_{\B_j}\left(\phi_{\ell}(\xi)\right)\right|^2\right)}
{\frac{1}{n}\sum_{s=0}^{n-1}\left|\widetilde{M}_{\mathcal{B}_{s}}(\xi)\right|^2}\\
&=&\sum_{\ell\in\mathcal{L}}\mathbf{M}_{\ell}(\xi)\\
&=&1 \quad (\text{by  Lemma } \ref{lem1}).
\end{eqnarray*}
Conversely, suppose that $\widetilde{Q}_{\Gamma}(\xi)\leq1$ for all $\xi\in\R$. From Lemma \ref{lem-Q}, we can obtain that
\begin{eqnarray*}
Q_{\Gamma}\left(\xi\right) &=& \left( \frac{1}{n}\sum_{s=0}^{n-1}
\left|{M}_{\mathcal{B}_{s}}
\left(\xi\right)\right|^2\right)
\widetilde{Q}_{\Gamma}
\left(\xi\right)\le \frac{1}{n}\sum_{s=0}^{n-1}
\left|{M}_{\mathcal{B}_{s}}
\left(\xi\right)\right|^2.
\end{eqnarray*}
Hence, from Lemma \ref{cprop2}, $\Lambda=\frac{{\mathcal L}_2}{N}\oplus\Gamma$ is an orthonormal set of $\mu_{N, \D}$.
Since ${Q}_{\Gamma}(\xi), \left|\widetilde{M}_{\mathcal{D}}\left(\phi_{\ell}(\xi)\right)\right|^2$ and $\left|\widetilde{M}_{\B_s}\left(\phi_{\ell}(\xi)\right)\right|^2$ are all entire functions on $\C$,  so is $\widetilde{Q}_{\Gamma}(\xi)$. This proves {\rm(i)}.

\medskip

A similar argument proves {\rm(ii)}  using Theorem \ref{fomula1}, instead of Lemma \ref{cprop2}.
\end{proof}


\section{Proof of Theorems  \ref{th2}, \ref{th3} and \ref{th5}}\label{sec.3}

This section proves Theorems   \ref{th2}, \ref{th3} and \ref{th5}. We begin with two basic  lemmas.

\begin{lem}\label{lem4.1}
Let $\gamma$ be an element in $\Z$. If $M_{\frac{\A\oplus\B_s}{N}}({\gamma})=0$ for each $s=0,1, \ldots, n-1$, then either $M_{\frac{\D}{N}}(\gamma)=0$ or $M_{\frac{\D}{N^2}}(\gamma)=0.$ 
\end{lem}
\begin{proof}
Since $\gamma\in\Z$ and each $M_{\B_s}$ is $1-$periodic, we have  $M_{\B_s}(\gamma)=1$.  Recall that $\D=\bigcup_{s=0}^{n-1}(a_s+N\B_s)$, we obtain  
$$M_{\frac{\D}{N}}\left(\gamma\right) = \frac1n\sum_{s=0}^{n-1}e\left(-\frac{a_s\gamma}{N}\right)M_{\B_s}(\gamma)=M_{\frac{\A}{N}}(\gamma),$$
and 
\begin{eqnarray*}
 M_{\frac{\D}{N^{2}}}\left(\gamma\right)=\frac1n\sum_{s=0}^{n-1}e\left(-\frac{a_s\gamma}{N^2}\right)M_{\frac{\B_s}{N}}\left(\gamma\right).
 \end{eqnarray*}
 The assertion follows from the assumption and $M_{\frac{\A\oplus\B_s}{N}}({\gamma})=M_{\frac{\A}{N}}({\gamma})M_{\frac{\B_s}{N}}({\gamma})$.
\end{proof}

\bigskip

\begin{lem}\label{prop.rep1}
Suppose that $\Lambda\subset \frac{\mathbb{Z}}{N}$ is a self-replicating translation set with respect to $(N,\widetilde{{\mathcal L}})$, that is,  $\Lambda=\widetilde{{\mathcal L}}\oplus N\Lambda$. Then it can be written as 
\begin{equation*}
  \Lambda=\frac{\mathcal{L}_2}{N}\oplus\Gamma\quad \text{where}\quad\Gamma=\mathcal{L}\oplus N\Gamma.
 \end{equation*}
Moreover, it is an orthonormal set of $\mu_{N,\mathcal{D}}$.
\end{lem}

\begin{proof} 
Denote $\Gamma=\mathcal{L}_1\oplus N\La$.  It follows from   $\widetilde{{\mathcal L}}={\mathcal L}_1\oplus \frac{{\mathcal L}_2}{N}$ that $ \Lambda=\frac{\mathcal{L}_2}{N}\oplus\Gamma$ and 
$$\Gamma={{\mathcal L}_1}\oplus N\Lambda={{\mathcal L}_1}\oplus N(\widetilde{{\mathcal L}}\oplus N\Lambda)={{\mathcal L}_1}\oplus {{\mathcal L}_2}\oplus N(\mathcal{L}_1+N\Lambda)={\mathcal L}\oplus N\Gamma.$$
The first assertion follows.

We prove the second assertion by showing that $\widehat{\mu}_{N,\mathcal{D}}(\lambda-\lambda')=0$ for any distinct elements $\lambda,\lambda'\in\Lambda$. Indeed, we can respectively write them as 
$$\lambda=\frac{\ell_0}{N}+\gamma,\quad \lambda'=\frac{\ell_0'}{N}+\gamma',$$
where $\ell_0,\ell_0'\in{\mathcal L}_2$ and $\gamma, \gamma'\in\Gamma\subset\Z$. We prove it by the following two cases. 

If $\ell_0\neq\ell'_0$, then $M_{\B_s}\left(\frac{\ell_0-\ell_0'}{N}\right)=0$, as $(N, \B_s, {\mathcal L}_2)$ forms a Hadamard triple for each $s$. 
Therefore, 
\begin{eqnarray*}
 M_{\frac{\D}{N}}(\lambda-\lambda')&=&\frac{1}{n}\sum_{s=0}^{n-1}e\left(-\frac{a_s}{N}(\lambda-\lambda')\right)M_{\B_s}\left(\lambda-\lambda'\right)\\
 &=&\frac{1}{n}\sum_{s=0}^{n-1}e\left(-\frac{a_s}{N}(\lambda-\lambda')\right)M_{\B_s}\left(\frac{\ell_0-\ell_0'}{N}\right)\\ 
 &=&0,
\end{eqnarray*}
 hence $\widehat{\mu}_{N,\mathcal{D}}(\lambda-\lambda')=\prod_{j=1}^{\infty}M_{\frac{\D}{N^j}}(\lambda-\lambda')=0$.

If $\ell_0=\ell'_0$, then $\gamma\neq\gamma'$.  Let $k\geq1$ be the smallest integer such that $\gamma\not\equiv\gamma'\pmod{N^k}.$  Observing the invariant equation 
 $\Gamma={\mathcal L}\oplus N\Gamma$ yields that 
$$\Gamma={\mathcal L}+N{\mathcal L}+\cdots+N^{k-1}{\mathcal L}+ N^k\Gamma. 
$$
Thus, there exist two sequences $({\ell_1}, {\ell_2}, \cdots, \ell_k), (\ell_1', \ell_2', \cdots, \ell'_k)\in \mathcal{L}^k$ and $\gamma_k, \gamma_k'\in\Gamma$ such that
 \begin{equation*}\label{eq12}
\gamma=\sum_{j=1}^{k}N^{j-1}\ell_j+N^k\gamma_k,\quad \gamma'=\sum_{j=1}^{k}N^{j-1}\ell_j'+N^k\gamma'_k,
 \end{equation*} 
where $({\ell_1}, {\ell_2}, \cdots, \ell_{k-1})=(\ell_1', \ell_2', \cdots, \ell'_{k-1})$ and $\ell_k\neq\ell_k'$. This gives  
 $$\frac{\lambda-\lambda'}{N^{k-1}}=\frac{\gamma-\gamma'}{N^{k-1}}=\ell_k-\ell_k'+N(\gamma_k-\gamma_k')\in\Z.$$
Since $(N, \A\oplus\B_s, {\mathcal L})$ forms a Hadamard triple for each $s=0, 1, \ldots, n-1$, it follows that  
 $$M_{\frac{\A\oplus\B_s}{N}}\left(\frac{\lambda-\lambda'}{N^{k-1}}\right)=M_{\frac{\A\oplus\B_s}{N}}\left(\ell_k-\ell_k'\right)=0.$$
By Lemma \ref{lem4.1}, we have 
 \begin{eqnarray*}
M_{\frac{\D}{N^k}}\left(\lambda-\lambda'\right)   \cdot M_{\frac{\D}{N^{k+1}}}\left(\lambda-\lambda'\right)=M_{\frac{\D}{N}}\left(\frac{\lambda-\lambda'}{N^{k-1}}\right)   \cdot M_{\frac{\D}{N^{2}}}\left(\frac{\lambda-\lambda'}{N^{k-1}}\right)=0,
 \end{eqnarray*}
 hence $\widehat{\mu}_{N,\mathcal{D}}(\lambda-\lambda')=\prod_{j=1}^{\infty}M_{\frac{\D}{N^j}}(\lambda-\lambda')=0$. The proof is complete. 
\end{proof}

\bigskip

\subsection{Proof of Theorem \ref{th2}}

The proof proceeds via the chain of implications: $\rm{(i)} \Rightarrow \rm{(ii)} \Rightarrow \rm{(iii)} \Rightarrow \rm{(i)}$.

\bigskip

\noindent\textbf{Proof of ${\rm(i)} \Rightarrow {\rm(ii)}$.} 
According to Lemma \ref{prop.rep1},   $\La=\frac{{\mathcal L}_2}{N}\oplus\Gamma$ is an orthonormal set of $\mu_{N, \D}$.  By Theorem \ref{fomula2} {\rm(i)},
$$\widetilde{Q}_{\Gamma}(\xi)\le1,\qquad\forall\ \xi\in\R.$$
 Suppose that $\La$ is not a spectrum of  $\mu_{N,\mathcal{D}}$. By Theorem \ref{fomula2} {\rm(ii)},  $\widetilde{Q}_{\Gamma}(\xi)<1$ for some $\xi\in\R$.    Let
$$X=\left\{\xi\in B_{r}: \widetilde{Q}_{\Gamma}(\xi)=\inf_{\eta \in B_{r}}\widetilde{Q}_{\Gamma}(\eta)\right\}.$$ 
Since $\Gamma$ is an orthonormal set of $\mu_{N, \D}$,  for each $\gamma\in \Gamma$, we have
\begin{equation*}
    \widetilde Q_{\Gamma}(-\gamma)=\sum_{\gamma_1\in\Gamma}\left|\widehat{\mu}_{N, \D}(-\gamma +\gamma_1)\right|^2=\left|\widehat{\mu}_{N, \D}(0)\right|^2=1.
\end{equation*} 
Moreover, $\widetilde{Q}_{\Gamma}$ can be extended to an entire function on $\mathbb{C}$ and $B_r:=[-r,r]$ is a compact set with infinite cardinality. 
Therefore, $\#X<\infty$ and  $-X\subset \Gamma^{\complement}$, the complement of $\Gamma$. 


Next, we apply the Ruelle transfer operator to derive the existence of \textit{cycle points} for ${\mathcal L}$.  We choose  $$r=\max\left\{|\xi|: \xi\in \frac{\mathcal L}{N-1} \cup T(N,\mathcal{L}) \right\},$$ it follows that  $T(N,\mathcal{L})\subset B_{r}$ and ${\mathcal L}\subset (N-1)B_r$. Remember that $\phi_\ell(\xi)=N^{-1}(\xi+\ell), \ell\in\mathcal{L}$. Thus,  for every $\ell\in\mathcal{L}$ and $\xi\in B_r$, we have
$$\left|\phi_{\ell}(\xi)\right|=\left|N^{-1}(\xi+\ell)\right|\leq N^{-1}\left(r+r(N-1)\right)=r,$$ i.e., $\phi_{\ell}\left(B_r\right)\subset B_r$ for all $\ell\in{\mathcal L}.$


Fixing $\xi_{0}\in X$ and setting $Y_{0}=\left\{\xi_{0}\right\}$. For each $k\geq1$, we define $Y_{k}$ inductively by
$$Y_{k}=\Big\{\xi_{k}=\phi_{\ell_{k-1}}(\xi_{k-1}): \xi_{k-1}\in Y_{k-1},\ell_{k-1}\in\mathcal{L} \;\text{such that}\; \phi_{\ell_{k-1}}(\xi_{k-1})\in X\Big\}.$$
By Lemma \ref{lem1}, $\sum_{\ell\in\mathcal{L}}\mathbf{M}_{\ell}(\xi)
=1$ for all $\xi\in\R$,  and hence for any $\xi_{k-1}\in Y_{k-1}$, we have 
$$
\widetilde{Q}_{\Gamma}(\xi_{k-1})
=\sum_{\ell\in\mathcal{L}}\mathbf{M}_{\ell}(\xi_{k-1})\widetilde{Q}_{\Gamma}\left(\phi_{\ell}(\xi_{k-1})\right)
\geqslant \widetilde{Q}_{\Gamma}(\xi_{k-1})\sum_{\ell\in\mathcal{L}}\mathbf{M}_{\ell}(\xi_{k-1})
= \widetilde{Q}_{\Gamma}(\xi_{k-1}).
$$  
Thus, 
 $$\mathbf{M}_{\ell}(\xi_{k-1})>0\quad \Rightarrow\quad \phi_{\ell}(\xi_{k-1})\in Y_k,$$ i.e., each $\xi_{k-1}\in Y_{k-1}$ has at least one offspring in $Y_k$.  
Hence all the sets $Y_k$ are non-empty. 

For any $\xi_k\in Y_k$, there exist $\ell_{0},\ell_{1},\ldots,\ell_{k-1}\in\mathcal{L}$ such that
$$\xi_k=\phi_{\ell_{k-1}}\circ\cdots\circ\phi_{\ell_{0}}(\xi_0)=\frac1{N^k}(\xi_0+\ell_0+\cdots+N^{k-1}\ell_{k-1}).$$
Since $(N, {\mathcal A}\oplus{\mathcal B}_s, {\mathcal L})$ is a Hadamard triple, the elements of ${\mathcal L}$ are distinct modulo $N$. Thus, for different $(\ell_0, \cdots, \ell_{k-1})\neq(\ell_0', \cdots, \ell_{k-1}')$, the associated points $\xi_k$ and $\xi_k'$ are different.  This yields that the elements in $Y_k$ are distinct. Therefore,  the cardinality of $Y_k$ is increasing in $k$.

Since $Y_{k}\subset X$ for all $k$ and $\#X<\infty$, there exists $k_0$ such that $\#Y_k$ is constant for all $k \ge k_0$. Thus, for $k\geq k_0$, each $\xi_k\in Y_k$ has only one offspring $\xi_{k+1}=\phi_{\ell_k}(\xi_k)$, i.e., there is only one $\ell_k\in{\mathcal L}$ such that $\mathbf{M}_{\ell_k}(\xi_{k})>0$, and $\mathbf{M}_{\ell}(\xi_{k})=0$ for all $\ell\in{\mathcal L}\setminus\{\ell_k\}$. Starting from  $\xi_{k_{0}}$, we obtain a sequence  $\{\xi_{k}\}_{k\geq k_{0}}$. Since $X$ is finite, there exist integers $k_1 \ge k_0$ and $m > 0$ such that, for all $k \ge k_1$,$$\xi_{k}=\xi_{k+m}=\phi_{\ell_{k+m-1}}\circ\cdots\circ\phi_{\ell_{k}}(\xi_{k}).$$
In other words,  $\{\xi_{k}\}_{k\geq k_{1}}\subset X$ is a sequence of cycle points for ${\mathcal L}$. Obviously, $-\xi_k\notin \Gamma$ for all $k\geq k_1.$  

Finally, we show that $\xi_{k}\in \mathbb{Z}$ for all $k\geq k_1$. Fix $k\geq k_1$. As  discussed above, $\ell_k$ is the unique element in ${\mathcal L}$ such that $\mathbf{M}_{\ell_k}(\xi_k)>0.$ Thus, by Lemma \ref{lem1},  
 \begin{equation*}
\mathbf{M}_{\ell_k}(\xi_k)
=\left|\widetilde{M}_{\mathcal{D}}\left(\phi_{\ell_k}(\xi_k)\right)\right|^2
\frac{\sum_{j=0}^{n-1}\left|M_{\mathcal{B}_{j}}\left(\phi_{\ell_k}(\xi_k)\right)\right|^2}
{\sum_{s=0}^{n-1}\left|\widetilde{M}_{\mathcal{B}_{s}}(\xi_k)\right|^2}=1,
 \end{equation*}
which gives 
\begin{align}\label{eq-0} 
\begin{split}
0<\sum_{s=0}^{n-1}\left|\widetilde{M}_{\mathcal{B}_{s}}(\xi_k)\right|^2&=\left|\widetilde{M}_{\mathcal{D}}\left(\phi_{\ell_k}(\xi_k)\right)\right|^2\cdot 
\sum_{j=0}^{n-1}\left|M_{\mathcal{B}_{j}}\left(\phi_{\ell_k}(\xi_k)\right)\right|^2\\
&=\left|\widetilde{M}_{\mathcal{D}}\left(\xi_{k+1}\right)\right|^2\cdot 
\sum_{j=0}^{n-1}\left|M_{\mathcal{B}_{j}}\left(\xi_{k+1}\right)\right|^2.    
\end{split}
\end{align}
Multiplying $F(e(-\xi_k))$ on both sides of \eqref{eq-0},  and applying  \eqref{eq.tra} gives 
 \begin{eqnarray}\label{eq-1}
0<\sum_{s=0}^{n-1}
\left|{M}_{\mathcal{B}_{s}}(\xi_{k})\right|^2&=& \left|{M}_{\mathcal{D}}
\left(\xi_{k+1}\right)\right|^2\cdot\sum_{j=0}^{n-1}\left|M_{\mathcal{B}_{j}}(\xi_{k+1})\right|^2\le n\cdot \left|{M}_{\mathcal{D}}
\left(\xi_{k+1}\right)\right|^2.
\end{eqnarray}
By Cauchy-Schwarz's inequality,
\begin{align}\label{eq-2}
\begin{split}
n\cdot \left|{M}_{\mathcal{D}}\left(\xi_{k+1}\right)\right|^2
&=\frac{1}{n}\left|\sum_{s=0}^{n-1}e(-a_s\xi_{k+1})M_{\B_s}(N\xi_{k+1})\right|^2\\
&=\frac{1}{n}\left|\sum_{s=0}^{n-1}e(-a_s\xi_{k+1})M_{\B_s}(\xi_{k}+\ell_k)\right|^2\\
&\leq \frac{1}{n}\left(\sum_{s=0}^{n-1}\left|e(-a_s\xi_{k+1})\right|^{2}\right)\left(\sum_{s=0}^{n-1}\left|M_{\B_s}(\xi_{k})\right|^{2}\right) \\
& = \sum_{s=0}^{n-1}\left|M_{\B_s}(\xi_{k})\right|^{2}.    
\end{split}
 \end{align}
 By comparing the left-hand side of  \eqref{eq-1} with the right-hand side of  \eqref{eq-2}, we conclude that the original inequalities must actually be equalities. 
In particular,  \eqref{eq-1} yields that $M_{\B_j}(\xi_{k+1})=1$ for each $j=0, 1, \cdots, n-1$, and  \eqref{eq-2} implies that the following two vectors
$$\Big(e(-a_0\xi_{k+1}), \cdots, e(-a_{n-1}\xi_{k+1})\Big),\quad \Big(M_{\B_0}(\xi_k), \cdots, M_{\B_{n-1}}(\xi_k)\Big)$$ are linearly dependent.  
Notice that $a_0=0$. Thus,  if $k\geq k_1+1,$ then 
 $$e\left(-a_{s}\xi_{k+1}\right)=M_{\mathcal{B}_{s}}\left(\xi_{k}\right)=1$$ for all $s=0, 1,\cdots, n-1$, which means that  $$M_{\mathcal{D}_{0}}(\xi_{k})
 =\frac{1}{\#\mathcal{D}_{0}}\sum_{d\in \mathcal{D}_{0}}e(-d\xi_{k})
 =\frac{1}{n}\sum_{s=0}^{n-1}e\left(-a_{s}\xi_{k}\right)M_{\mathcal{B}_{s}}\left(\xi_{k}\right)=1.$$  Thus, $d\xi_{k} \in\Z$ for all $d\in \mathcal{D}_{0}$. It follows from $0\in\mathcal{D}_{0}$ and $\gcd{(\mathcal{D}_{0})}=1$ that $\xi_{k}\in\mathbb{Z}$ for all $k\geq k_1+1$.   Since  $\xi_{k}=\xi_{k+m}$  for all $k\geq k_1$, we obtain $\xi_{k}\in\mathbb{Z}$ for all $k\geq k_1$.

In summary, we have proved that 
 $-\xi_{k}\in \mathbb{Z}\setminus \Gamma$ for all $k\geq k_1$, (ii) follows.  \qed

\bigskip
\noindent{\bf Remark.}  
If an integral cycle point $\xi$ is the fixed point of the map $\phi_{\ell_{m}}\circ\cdots\circ\phi_{\ell_1}$, that is, 
$$
\xi=\phi_{\ell_{m}}\circ\cdots\circ\phi_{\ell_1}(\xi)=\sum_{k=1}^m\frac{\ell_{m-k+1}}{N^k}+\frac{\xi}{N^m}.
$$
After suitable adjustments and an induction argument, we obtain
\begin{equation*}
\label{eq-cycle}-\xi=\sum_{j=1}^{m}N^{j-1}\ell_{j}+N^m(-\xi)=\sum_{k=0}^{\infty}N^{mk}\left(\sum_{j=1}^{m}N^{j-1}\ell_{j}\right)=:\sum_{k=1}^{\infty}N^{k-1}\ell_k,
\end{equation*}
where $\ell_k=\ell_{k\pmod{m}}$ for all $k\geq1$. That is, $-\xi$ admits an $N$-adic periodic expansion with digits in ${\mathcal L}$. 

\bigskip

The following Proposition \ref{prop-orth} is used  to prove the implication   ${\rm(ii)} \Rightarrow {\rm(iii)}$ in Theorem \ref{th2}. 
\begin{prop}\label{prop-orth}
Let $\xi$ be an integral cycle point for ${\mathcal L}$ and $\gamma\in\Gamma$. If $-\xi\neq\gamma$ and $k$ is the smallest integer such that  $-\xi\not\equiv\gamma\pmod{N^k}$, then 
$$M_{\frac{\A\oplus\B_s}{N^k}}(\xi+\gamma)=0,\quad \forall s=0, 1, \ldots, n-1.$$ 
\end{prop}
\begin{proof} According to the above Remark, $-\xi$ admits an $N$-adic periodic expansion with digits in ${\mathcal L}$
$$-\xi=\sum_{j=1}^{\infty}N^{j-1}\ell_j \quad \text{with} \quad \ell_j\in{\mathcal L}.$$
Because invariance $\Gamma={\mathcal L}\oplus N\Gamma$ implies  that $\Gamma={\mathcal L}+N{\mathcal L}+\cdots+N^{k-1}{\mathcal L}+ N^k\Gamma$, then $\gamma\in\Gamma$ can be represented as  
$$\gamma=\sum_{j=1}^kN^{j-1}\ell_j'+N^k\gamma_{k},$$ where 
$\ell_j'\in\mathcal{L}$ for all $1\leq j\leq k$ and $\gamma_{k}\in\Gamma$. 
As $k$ is the smallest integer such that $-\xi\not\equiv\gamma\pmod{N^k}$, we have
$$(\ell_1, \cdots, \ell_{k-1})=(\ell_1', \cdots, \ell_{k-1}') \quad\text{ and }\quad  \ell_k\neq\ell_k'.$$ Thus, there is an integer $t\in\Z$ such that   
$$\frac{\gamma+\xi}{N^{k-1}}=\frac{\gamma-(-\xi)}{N^{k-1}}=\ell_k'-\ell_k+Nt.$$
Since $(N, \A\oplus\B_s, {\mathcal L})$ forms a Hadamard triple for each $s=0, 1, \ldots, n-1$, it follows that 
 $$M_{\frac{\A\oplus\B_s}{N^{k}}}\left(\xi+\gamma\right)=M_{\A\oplus\B_s}\left(\frac{\ell_k'-\ell_k}{N}\right)=0.$$ The proof of Proposition \ref{prop-orth} is complete. 
\end{proof}

\bigskip

\noindent\textbf{Proof of ${\rm(ii)} \Rightarrow {\rm(iii)}$.} Let $\xi$ be an integral cycle point for $\mathcal{L}$ such that $-\xi\notin \Gamma$. We will prove assertion (iii) by showing that $\xi\in\mathcal{W}(\widehat{\mu}_{N, \mathcal{D}})$. Because of $\xi\in\Z$, we have 
$$\frac1n\sum_{s=0}^{n-1}|M_{\B_s}(\xi)|^2=1>0,$$
that is, $\xi\in{\mathcal O}.$
If $\gamma\in\Gamma$, then $-\xi\neq\gamma$. Let $k$ be the smallest integer such that $-\xi\not\equiv\gamma\pmod{N^k}$.  By Proposition \ref{prop-orth}, 
$$M_{\frac{\A\oplus\B_s}{N^k}}(\xi+\gamma)=0
$$ for each $s=0, 1, \ldots, n-1$. 
By Lemma \ref{lem4.1}, we have  
 \begin{eqnarray*}
M_{\frac{\D}{N^k}}\left(\xi+\gamma\right)   \cdot M_{\frac{\D}{N^{k+1}}}\left(\xi+\gamma\right)=0,
 \end{eqnarray*}which yields that $\widehat{\mu}_{N,\mathcal{D}}(\xi+\gamma)=\prod_{j=1}^{\infty}M_{\frac{\D}{N^j}}(\xi+\gamma)=0$. Since $\gamma\in\Gamma$ can be chosen arbitrarily, we obtain $\xi\in\mathcal{W}(\widehat{\mu}_{N, \mathcal{D}})$.  Hence, {\rm(iii)} follows.
 \qed

\bigskip

\noindent\textbf{Proof of ${\rm(iii)} \Rightarrow {\rm(i)}$.}
Suppose that $\mathcal{W}(\widehat{\mu}_{N, \mathcal{D}})\neq\emptyset$. By the definition of $\mathcal{W}(\widehat{\mu}_{N, \mathcal{D}})$, there exists $\xi_{0}\in \mathcal{O} \subset  B_{r}$ such that  $$\widehat{\mu}_{N,\mathcal{D}}(\xi_{0}+\gamma)=0, \qquad \forall\gamma\in \Gamma,$$ and hence  $Q_{\Gamma}(\xi_{0})=0$.
However, $\xi_{0}\in\mathcal{O}$ implies that $$\frac{1}{n}\sum_{s=0}^{n-1}\left|M_{\mathcal{B}_{s}}(\xi_{0})\right|^{2}>0.$$ Therefore, the above element $\xi_{0}\in B_{r}$  satisfies 
 $$Q_{\Gamma}(\xi_{0})<\frac{1}{n}\sum_{s=0}^{n-1}
 \left|M_{\mathcal{B}_{s}}(\xi_{0})\right|^{2}.$$  By Theorem \ref{fomula1}, $\La=\frac{{\mathcal L}_2}{N}\oplus\Gamma$ is not a spectrum of $\mu_{N, \D}$, {\rm(i)} is proved.
\qed

\medskip

\subsection{Proof of Theorem \ref{th3}}

In addition to Theorem \ref{th2}, the proof of Theorem \ref{th3} relies on Theorem \ref{cth4} given by Dutkay and Lai \cite{DL2017}.
Denote $$\mathcal{C}=\left\{c:c \;\text{is an integral cycle point for }\mathcal{L}\right\}.$$  

\begin{theo}[\!\cite{DL2017}]\label{cth4}
Let $N>1$ be an integer, and let $\mathcal{D}, \mathcal{L}\subset\mathbb{Z}$ with $0\in \mathcal{D}\cap \mathcal{L}$ and $\gcd(\mathcal{D})=1$. Suppose that $(N,\mathcal{D},\mathcal{L})$ forms a Hadamard triple, and let $\mu_{N,\mathcal{D}}$ be the measure generated by the IFS $\{N^{-1}(x+d):d\in\mathcal{D}\}$. Then 
$$\Gamma_0: =\Big\{\ell_0+\cdots+N^{n-1}\ell_{n-1}+N^{n}(-c): \ell_0,\dots, \ell_{n-1}\in{\mathcal L}, c \in {\mathcal C}, n\geq0\Big\}$$ is a spectrum of $\mu_{N,\mathcal{D}}$.
\end{theo}

\begin{lem}\label{cth5}
Suppose that $0\in \Gamma\subset\Z$. Then $\Gamma=\Gamma_0$ if and only if $\,\Gamma$ satisfies the following conditions $$-{\mathcal C}\subset\Gamma \quad  \text{and}\quad  \Gamma={\mathcal L}+N\Gamma.$$
\end{lem}
\begin{proof} 
 $(\Rightarrow)$ It follows immediately from the definition of $\Gamma_0$. 

$(\Leftarrow)$ It has been proved in \cite[Lemma 4.1]{DL2017} that $\Gamma_{0}$ is the smallest set that satisfies $-{\mathcal C}\subset\Gamma_0$ and ${\mathcal L}+N\Gamma_0\subset\Gamma_0$. 
Thus, it  is enough to prove  $\Gamma\subset\Gamma_0$. 
In fact, for any $\gamma_0\in\Gamma$, the invariance  $\Gamma={\mathcal L}+N\Gamma$ implies that $\gamma_0=\ell_0+N\gamma_1$ for some $\ell_0\in{\mathcal L}$, $\gamma_1\in\Gamma$. Similarly, $\gamma_1=\ell_1+N\gamma_2$ for some $\ell_1\in{\mathcal L}$, $\gamma_2\in\Gamma$. Iterating this process, we obtain sequences
$\left\{\ell_k\right\}_{k=0}^{\infty}\subset \mathcal{L}$ and $\left\{\gamma_k\right\}_{k=0}^{\infty}\subset\Gamma$ such that 
$$\gamma_k=\ell_k+N\gamma_{k+1},\qquad k\geq 0.$$
Thus,  for each $k\geq0,$ one has 
$$-\gamma_{k+1}=\frac{\ell_k-\gamma_k}{N}=\cdots=\sum_{j=1}^{k+1}\frac{\ell_{k+1-j}}{N^{j}}-\frac{\gamma_0}{N^{k+1}}\in \left(T(N, {\mathcal L})-\frac{\gamma_0}{N^{k+1}}\right)\cap\Z\subset B_{|\gamma_0|}(T(N,\mathcal{L}))\cap\Z,$$ 
where $B_{|\gamma_0|}(T(N,\mathcal{L}))$ is the $|\gamma_0|-$neighborhood of the set $T(N, {\mathcal L})$. 
Since $B_{|\gamma_0|}(T(N,\mathcal{L}))\cap\Z$ is a finite set, 
there exist $k\geq0$ and $i\geq1$ such that $\gamma_{k+i}=\gamma_{k}$, that is, 
$$-\gamma_k=-\gamma_{k+i}=\sum_{j=1}^{i}\frac{\ell_{k+i-j}}{N^j}-\frac{\gamma_{k}}{N^{i}}=\phi_{\ell_{k+i-1}}\circ\cdots\circ\phi_{\ell_k}(-\gamma_k).$$ Therefore, 
$-\gamma_k$ is the fixed point of the map $\phi_{\ell_{k+i-1}}\circ\cdots\circ\phi_{\ell_k}$, i.e.,  $\gamma_k\in -\mathcal C,$ and hence   
$$\gamma_0=\ell_0+N\ell_1+\cdots+N^{k-1}\ell_{k-1}+N^{k}\gamma_{k}\in\Gamma_0.$$
This shows $\Gamma\subset\Gamma_0.$
\end{proof}

\bigskip

We are going to prove Theorem \ref{th3}.

\noindent
{\bf Proof of Theorem \ref{th3}.}
${\rm(i)}$  Suppose that there exists $s_0\in\{0,1,\ldots,n-1\}$ such that $\left(\mu_{N,\mathcal{A}\oplus\mathcal{B}_{s_0}}, \Gamma\right)$ forms a spectral pair,  but  $\La$ is not a spectrum of $\mu_{N,\mathcal D}$. By Theorem \ref{th2}  ${\rm(ii)}$, there exists an integral cycle point $\xi$ for $\mathcal{L}$ such that $-\xi\notin\Gamma$. Thus, for any $\gamma\in\Gamma$,  we have $\gamma\neq -\xi$. 
Let $k$ be the smallest integer such that $\gamma\not\equiv-\xi\pmod{N^k}$. We obtain, from Proposition \ref{prop-orth}, that  
$$M_{\frac{\A\oplus\B_{s_0}}{N^k}}(\gamma+\xi)=0.$$  Hence,
\begin{equation*}\widehat{\mu}_{N, \A\oplus\B_{s_{0}}}(\xi+\gamma)=0 \quad \text{for all}~ \gamma\in\Gamma,
\end{equation*}
which means that $-\xi$ is orthogonal to $\Gamma$ in $L^2(\mu_{N,\mathcal{A}\oplus\mathcal{B}_{s_0}})$. This contradicts  the assumption that $\Ga$ is a spectrum of the measure $\mu_{N,\mathcal{A}\oplus\mathcal{B}_{s_0}}$.

${\rm(ii)}$  Suppose that  $\Lambda=\frac{{\mathcal L}_2}{N}\oplus\Gamma$ is a spectrum of $\mu_{N,\mathcal{D}}$,  where  $\Gamma=N\Gamma\oplus\mathcal{L}\subset\mathbb{Z}$. By Theorem \ref{th2}, we have $-{\mathcal C}\subset\Gamma$. Now Lemma \ref{cth5} implies that  $\Gamma=\Gamma_0$. Therefore, by  Theorem \ref{cth4}, $\Gamma$ is a spectrum of $\mu_{N,\mathcal{A}\oplus\mathcal{B}_s}$ for each
$s$ with $\gcd(\mathcal{A}\oplus\mathcal{B}_s)=1$.
\qed

\bigskip

  

\subsection{Proof of Theorem \ref{th5}}

The proof of Theorem \ref{th5} requires a detailed analysis of the set  $${\mathcal C}=\{\xi: \xi \text{ is an integral cycle point for }{\mathcal L}\}.$$  
The set ${\mathcal C}$ was first studied by {\L}aba and Wang  \cite{LW2002}, which showed that ${\mathcal C}$ must be contained in the set $T(N,{\mathcal L})\cap\Z$. Subsequently,  the authors of \cite{DH2016} established the identity  $\mathcal{C}=T(N, {\mathcal L})\cap\Z$ for the digit set 
$\mathcal{L}:=\{0,p\},p\in 2\mathbb{Z}+1,$ when they studied the spectra of the middle-fourth Cantor measure $\mu_{4, \{0, 2\}}$. 
In this section,  we prove that this identity holds in general.

\medskip

 \begin{prop} \label{lem.cy}
      $\mathcal{C}=T(N, {\mathcal L})\cap\Z.$ 
 \end{prop}

\begin{proof} Denote $T:=T(N, {\mathcal L})$.
We first prove that $\mathcal{C}\subset T\cap\Z$. 
Assume that $\xi_{0}\in \mathcal{C}.$ Then there exist $m\in\mathbb{N}$, $\ell_{0},\ldots,\ell_{m-1}\in\mathcal{L}$ such that
$$\xi_{0}=\xi_{m}=\phi_{\ell_{m-1}}\left(\xi_{m-1}\right)=\cdots=\phi_{\ell_{m-1}}\circ\phi_{\ell_{m-2}}\circ\cdots\circ\phi_{\ell_{0}}\left(\xi_{0}\right),$$
i.e.
$$\xi_{0}=\frac{\xi_{0}}{N^{m}}+\frac{\ell_{0}}{N^{m}}+\cdots+\frac{\ell_{m-1}}{N}.$$
By iterating this equality indefinitely, we obtain the  $N$-adic expansion of $\xi_{0}$. Hence, $\xi_{0}\in T\cap\mathbb{Z},$  so $\mathcal{C}\subset T\cap\Z$.

We now prove the reverse inclusion. Taking any $\xi_0\in T\cap\Z$. Since $T=\bigcup_{\ell\in{\mathcal L}}\phi_{\ell}(T)$,
there exist $\ell_{1}\in\mathcal{L}$ and $\xi_{1}\in T$ such that
$$\xi_0= \phi_{\ell_{1}}\left(\xi_{1}\right), \quad\text{hence}\quad \xi_{1}=N\xi_0-\ell_{1}\in T\cap\Z.$$
By iterating this construction, we obtain sequences $\{\xi_{i}\}_{i=0}^\infty\subset T\cap\Z$ and $\{\ell_{i+1}\}_{i=0}^\infty\subset \mathcal{L}$ such that 
$$\xi_{i}=\phi_{\ell_{i+1}}\left(\xi_{i+1}\right),\quad  \forall i\geq 0.$$   Since $T$ is compact,  there exists $i,j\in\N$ with  $i<j$ such that
  $$\xi_{j}=\xi_{i}=\phi_{\ell_{i+1}}\circ\cdots\circ\phi_{\ell_j}(\xi_j),$$ 
 i.e.,  $\xi_i\in{\mathcal C}$.  From
 $$\xi_{i-1}=\phi_{\ell_i}(\xi_i)=\frac{\xi_i+\ell_i}{N}\in\Z\quad \text{and}\quad \xi_{j-1}=\phi_{\ell_j}(\xi_j)=\phi_{\ell_{j}}(\xi_i)=\frac{\xi_i+\ell_{j}}{N}\in\Z,$$ we get 
$\ell_{i}\equiv-\xi_i\pmod{N}$ and $\ell_{j}\equiv-\xi_i\pmod{N}$. 
 As the elements of ${\mathcal L}$ lie in distinct residue modulo $N$, it follows that  $\ell_{i}=\ell_j$.  Hence, 
  $$\xi_{j-1}=\xi_{i-1}=\phi_{\ell_{i}}\cdots\phi_{\ell_{j-1}}(\xi_{j-1})\in {\mathcal C}.$$
  By induction, we obtain $\xi_{0}\in {\mathcal C}$.   The proof is finished.
\end{proof}

\bigskip  

Thus, by Proposition \ref{lem.cy}, in order to verify the criterion given in Theorem \ref{th2} {\rm(ii)}, it is enough to examine the integers contained in the compact set $T(N, {\mathcal L}).$ Notice that 
$$T(N, {\mathcal L})\subset\left[\frac{\ell_{\text{min}}}{N-1}, \frac{\ell_{\text{max}}}{N-1}\right],$$
where $\ell_{\text{min}}, \ell_{\text{max}}$ denote the smallest and the largest elements of ${\mathcal L}$, respectively.  Therefore, we only need to verify finitely many integers. 
Now we can prove Theorem \ref{th5} by using Theorem \ref{th2} and Proposition \ref{lem.cy}.

\bigskip

\noindent{\bf Proof of Theorem \ref{th5}.} 
By Lemma \ref{lem.ha} \rm{(ii)}, the elements in the direct sum $\mathcal{L}_1\oplus \mathcal{L}_2$  are in distinct residue modulo $N$, so $N\geq\#{\mathcal L}_1\cdot\#{\mathcal L}_2\geq4.$ From Lemma \ref{lem.ha} \rm{(i)}, we can replace  ${\mathcal L}_1, {\mathcal L}_2$ by 
$$\widehat{\mathcal L}_1\subset\{-1, 0, 1, \ldots, N-2\},\quad \widehat{\mathcal L}_2\subset\{2-N,  \ldots, -1, 0, 1\}$$ with  $0\in\widehat{\mathcal L}_1\cap\widehat{\mathcal L}_2$, and $(N, {\mathcal{A}}, \widehat{\mathcal L}_1)$, $(N, \mathcal{B}_s,\widehat{\mathcal L}_2)$, $(N, \mathcal{A}\oplus \mathcal{B}_s, \widehat{\mathcal L}_1\oplus \widehat{\mathcal L}_2)$ are  Hadamard triples for all $s=0,1,\ldots,n-1$.

If neither $(-1, 2-N)$ nor $(N-2, 1)$  lies in $\widehat{\mathcal L}_1\times \widehat{\mathcal L}_2$, then 
$$\widehat{\mathcal L}:=\widehat{\mathcal L}_1\oplus\widehat{\mathcal L}_2\subset[2-N, N-2].$$
If either $(-1, 2-N)$ or $(N-2, 1)$ lies in $\widehat{\mathcal L}_1\times \widehat{\mathcal L}_2$,  we may assume without loss of generality that  $(-1, 2-N)$ does.  Since $0\in\widehat{\mathcal L}_1\cap\widehat{\mathcal L}_2$ and $\widehat{\mathcal L}_1\oplus\widehat{\mathcal L}_2$ is a  direct sum, it follows that  $N-2, N-3\not\in\widehat{\mathcal L}_1$. 
Since $2-N \equiv 2 \pmod{N}$, we replace
 $2-N$ by $2$ in  $\widehat{\mathcal L}_2$. Thus,  we further assume that 
$$\widehat{\mathcal L}_1\subset\{-1, 0, 1, \dots, N-4\},\quad \widehat{\mathcal L}_2\subset\{3-N,  \dots,  0, 1, 2\}.$$ This also gives $$\widehat{\mathcal L}:=\widehat{\mathcal L}_1\oplus\widehat{\mathcal L}_2\subset[2-N, N-2].$$

Consequently, the self-similar set $T(N, \widehat{\mathcal L})\subset\left[\frac{2-N}{N-1}, \frac{N-2}{N-1}\right]$ contains no integers other than $0$. 
By  Proposition \ref{lem.cy}, there are no non-zero integral cycle points for $\widehat{\mathcal L}$.  Let $\widetilde{\mathcal L}=\widehat{\mathcal L}_1\oplus \frac{\widehat{\mathcal L}_2}{N}$.  
Since $0\in\Lambda(N, \widehat{\mathcal L})=\widehat{\mathcal L}\oplus N\Lambda(N, \widehat{\mathcal L})$, it follows from Theorem \ref{th2} that 
$$\Lambda(N, \widetilde{\mathcal L})=\frac{\widehat{\mathcal L}_2}{N}+\Lambda(N, \widehat{\mathcal L})$$
is a spectrum of $\mu_{N, \D}$.
\qed


\section{Examples and the dual spectral set conjecture}\label{sec.ex}

In this section, we provide examples to illustrate how one can generate a self-replicating spectrum using product-form Hadamard
triples. Then we  prove Theorem \ref{thm.dua.}.


\subsection{Examples}

The following example shows that some digit set is not in product-form, but after multiplying some factor, the resulting new digit set can form a product-form Hadamard triple.


\begin{exam}
Let $N=24=2^3\cdot 3$ and the digit set 
$$\mathcal{D}=\{0, 1, 16, 49\}=\{0, 1, 2^4, 1+2^4\cdot3\}.$$
Then 
$$3\D=24\{0, 2\}\cup \left(3+24\{0, 6\}\right)$$
is a product-form digit set generated by  Hadamard triples $(24, \mathcal{A}, \mathcal{L}_1)$ and $(24, \mathcal{B}_s, \mathcal{L}_2), s=1,2$, where  
 $$\mathcal{A}=\{0,3\}, \mathcal{B}_1=\{0, 2\}, \mathcal{B}_2=\{0, 6\}, {\mathcal L}_1=\{0, 12\} ~\text{and} ~{\mathcal L}_2=\{0, 6\},$$ and the corresponding digit set $\D_0=\{0, 2\}\cup \left(3+\{0, 6\}\right)$ satisfies $\gcd(\mathcal D_0) = 1$. Note that 
$$T\left(24, \{0, 6\}\oplus\{0, 12\}\right)\subset \left[0, \frac{18}{23}\right]$$
contains no integers other than $0$.  By Proposition \ref{lem.cy} and  Theorem \ref{th2} {\rm(ii)}, $\mu_{N, 3\D}$ has a spectrum $\frac{1}{4}\{0, 1\}\oplus\Lambda(24, \{0, 6\}\oplus\{0, 12\})$.  Dividing by 3, $\mu_{N, \D}$ has a spectrum $\frac{3}{4}\{0, 1\}\oplus 3\Lambda(24, \{0, 6\}\oplus\{0, 12\})$.
\end{exam}


The condition "$\gcd(\A\oplus\B_s)=1$" in Theorem \ref{th3} {\rm(ii)} cannot be omitted. Here we give a product-form Hadamard triple $(N, \D, {\mathcal L}_1\oplus{\mathcal L}_2)$  such that $\Lambda=\frac{\mathcal L_2}{N}\oplus\Gamma$ is a spectrum of $\mu_{N, \D}$, yet $(\mu_{N,\A\oplus\B_s}, \Gamma)$ does not form a spectral pair for each $s=0, 1, \dots, n-1$.


\begin{exam}
Let $N=16$ and  $$\mathcal{D}=16\{0,3\} \cup \left(30+16 \{0,5\}\right), \qquad  {\mathcal L}_1\oplus{\mathcal L}_2=\{0, 12\}\oplus\{0, 40\}.$$  
Then $(16, \D, {\mathcal L}_1\oplus{\mathcal L}_2)$  forms a product-form Hadamard triple with the digit sets $\A=\{0, 30\}$, $\B_1=\{0, 3\}, \B_2=\{0, 5\}$. Let $\Gamma:=\Gamma_0$, where $\Gamma_0$ is given by Theorem \ref{cth4}.   As the corresponding digit set $\D_0=\{0, 3, 30, 35\}$ satisfies $\gcd(\mathcal D_0) = 1$, by Theorem \ref{th2}, $\frac{\mathcal L_2}{16}\oplus\Gamma$   is a spectrum of $\mu_{N, \D}$.

On the other hand, consider the self-similar spectral measure $\mu_{N,\mathcal{A}\oplus\mathcal{B}_{1}}$, where $$\gcd(\mathcal{A}\oplus\mathcal{B}_{1})=\gcd(\{0,3,30,33\})=3>1.$$
Since $\frac83$ is the fixed point of the map $\phi_{40}(\xi)=\frac1{16}(\xi+40)$, it follows that $\frac{8}{3}$ is a cycle point for ${\mathcal L}$. 
We can check immediately that $e(-\frac83)$ is orthogonal to $E(\Gamma)$ in $L^{2}(\mu_{N, \A\oplus\B_1})$.  Hence $\Gamma$ cannot be a spectrum of  $\mu_{N,\mathcal{A}\oplus\mathcal{B}_{1}}$.

Similarly, for the self-similar spectral measure $\mu_{N,\mathcal{A}\oplus\mathcal{B}_{2}}$, $$\gcd(\mathcal{A}\oplus\mathcal{B}_{2})=\gcd(\{0,5,30,35\})=5>1.$$
By the same argument as above, $\frac45$ is the fixed point of the  map $\phi_{12}(\xi)=\frac1{16}(\xi+12)$ and $e(-\frac45)$ is orthogonal to $E(\Gamma)$ in $L^{2}(\mu_{N, \A\oplus\B_2})$. Hence $\Gamma$ cannot be a spectrum of  $\mu_{N,\mathcal{A}\oplus\mathcal{B}_{2}}$.
\end{exam}

\subsection{The dual spectral set conjecture.} 
In this subsection, we prove that the  dual spectral set conjecture holds for a self-replicating translation set  (Theorem \ref{thm.dua.}). 

The concept of density is very important for characterizing the tiling and spectrum.

\begin{defi}\label{density}
The {\it lower} and {\it upper densities} of ${\mathcal J}$ in $\Z$ are defined, respectively,  as 
    $$\underline{D}({\mathcal J})=\underset{h\to\infty}{\underline{\lim}}\frac{\#({\mathcal J\cap [-h, h]})}{2h},\qquad \overline{D}({\mathcal J})=\overline{\lim_{h\to\infty}}\frac{\#({\mathcal J\cap [-h, h]})}{2h}.$$
If these are equal,  we refer to the common value as the {\it density} of ${\mathcal J}$, denoted  by ${D}({\mathcal J})$.
\end{defi}
The proof of Theorem \ref{thm.dua.} depends on the structure theorems of one-dimensional self-similar tiles established by Lagarias and Wang \cite{LW1996} and Li and Rao \cite{LR2025},  as well as the periodicity of  translational tilings and spectra in dimension one discovered by Lagarias and Wang \cite{LW1996-1}, Lau and Rao \cite{LR03}  and  Iosevich and Kolountzakis \cite{IK2013}, respectively.  
 

\begin{proof} [\rm\textbf{{Proof of Theorem \ref{thm.dua.}}}]
Since the properties of tiling  and spectra are invariant under the translation and scaling of the digit set $\D$, we may assume that $0\in\D$ and $\gcd(\D)=1$.  Let $\Lambda\subset\Z$ be the set such that $\Lambda=\D+N\Lambda$. 

\noindent{\bf Claim.} If $\Lambda$ is periodic, then $T(N, \D)$ is a self-similar tile.

In fact, by \cite[Theorem 1.1]{LW1996-1}, it suffices to prove that for each $m\geq1$, all $N^m$ expansions in ${\bf D}_m=\sum_{j=0}^{m-1}N^j\D$ are distinct. Suppose on the contrary that $\#{\bf D}_m\le N^m-1$ for some $m\geq2$. Then $\#{\bf D}_{km}\le (N^m-1)^k$ for all $k\geq 1$. Let $R$ be a period of $\Lambda$, then there is a nonempty set $\A\subset\{0, 1, \dots, R-1\}$ such that 
$$\Lambda={\mathcal A}+R\Z,$$ and hence $$D(\Lambda)=\frac{\#\A}{R}>0.$$
On the other hand, it follows from $\Lambda=\D+N\Lambda$ that 
$$\Lambda={\bf D}_{km}+N^{km}\Lambda={\bf D}_{km}+N^{km}\A+N^{km}R\Z,$$
we conclude that 
$$D(\Lambda)\le\lim_{k\to\infty}\frac{(\#{\bf D}_{km})\cdot(\#\A)}{N^{km}R}\le \lim_{k\to\infty}\frac{(N^m-1)^k\cdot(\#\A)}{N^{km}R}=0.$$
It is a contradiction. Therefore,  $T(N, \D)$ is a self-similar tile. The claim is proved.

{\rm (i)} Suppose that $\Lambda$ is a spectrum of the set $\Omega\subset\R$. Thus, $\Lambda$ is periodic, since any spectrum of one-dimensional spectral set must be periodic (\!\cite[Theorem 1.5]{IK2013}). Now the above claim implies that   $T(N, \D)$ is a self-similar tile.  Since  $\Lambda=\D+N\Lambda\subset\Z$,  it follows from   \cite[Theorem 1.2]{LR03} that $\Lambda$ is a tiling set of $T(N, \D)$. 


{\rm (ii)} Suppose that $\Lambda$ is a tiling set of a translational tile  $\Omega'\subset\R$. Thus, $\Lambda$ is periodic, since any tiling set of one-dimensional translational tile must be periodic (\!\cite[Theorem 1]{LW1996}). As is argued in {\rm(i)},  $(T(N, \D),\Lambda)$ is a tiling pair. 
Therefore, by \cite[Theorem 1.2]{LR2025}, there exists  $m\geq1$ such that ${\bf D}_m=\D\oplus\cdots\oplus N^{m-1}\D$ is a skew-product-form digit set 
$${\bf D}_m = \bigcup_{s=0}^{n-1} (a_s + N^m \B_s)$$
where $\A \oplus \B_s$ is a complete residue set modulo $N^m$ for all $s = 0, 1, \ldots, n-1$ and 
$$\Lambda=\A\oplus N^m\Z.$$
According to the assumption of {\rm(ii)}, $\A$ is a spectral set in the cyclic group  $\Z_{N^m}:=\Z/N^m\Z$. Thus,  there is a finite set ${\mathcal L}_{\A}\subset\Z$ such that $(\A, \frac{1}{N^m}{\mathcal L}_A)$ forms a spectral pair. Clearly, $\left([0, \frac1{N^m}],N^m\Z\right)$ and $(\frac{1}{N^m}{\mathcal L}_A,\A)$ are both spectral pairs. Therefore, $\La$ is a spectrum of the set $\Omega=[0, \frac{1}{N^m}]+\frac{1}{N^m}{\mathcal L}_{\A}.$  
The proof of Theorem \ref{thm.dua.} is complete. 
\end{proof}

\noindent\textbf{Data availability} There is no data associated with this manuscript.

\noindent\textbf{\large Declarations}

\noindent\textbf{Conflict of interest} On behalf of all authors, the corresponding authors state that there is no Conflict of interest.

\end{document}